\documentclass[11pt]{amsart}
\usepackage{lmodern} 
\usepackage{amssymb,amsmath,amsthm,verbatim,enumitem}
\usepackage{mathtools}
\usepackage{fullpage}
\usepackage[dvipsnames]{xcolor}
\usepackage{url}
\definecolor{darkgreen}{RGB}{20,160,20}
\usepackage[pdftex]{hyperref}
\hypersetup{
    colorlinks=true,
    linkcolor=blue,
    citecolor=darkgreen,      
    urlcolor=black,
    linktocpage=true
}
\usepackage{array} 
\usepackage{mathrsfs}
\usepackage[all]{xy}
  \SelectTips{cm}{10}
  \everyxy={<2.5em,0em>:}
\usepackage{tikz}
\usepackage{tikz-cd}
\usetikzlibrary{calc}
\usetikzlibrary{decorations.pathmorphing}

\tikzset{curve/.style={settings={#1},to path={(\tikztostart)
    .. controls ($(\tikztostart)!\pv{pos}!(\tikztotarget)!\pv{height}!270:(\tikztotarget)$)
    and ($(\tikztostart)!1-\pv{pos}!(\tikztotarget)!\pv{height}!270:(\tikztotarget)$)
    .. (\tikztotarget)\tikztonodes}},
    settings/.code={\tikzset{quiver/.cd,#1}
        \def\pv##1{\pgfkeysvalueof{/tikz/quiver/##1}}},
    quiver/.cd,pos/.initial=0.35,height/.initial=0}
\usepackage{picinpar} 
\usepackage{stmaryrd}

\usepackage[new]{old-arrows}
\usepackage{fancyhdr}
  \DeclareFontFamily{OT1}{pzc}{}
  \DeclareFontShape{OT1}{pzc}{m}{it}{<-> s * [1.100] pzcmi7t}{}
  \DeclareMathAlphabet{\mathpzc}{OT1}{pzc}{m}{it}

\usepackage[OT1]{fontenc}

\usepackage[english]{babel}

\usepackage[utf8]{inputenc}

\usepackage{amsmath,amssymb,amsfonts,mathrsfs}

\usepackage{graphicx}

\usepackage{pdfpages}

\usepackage{csquotes}
\usepackage[
    backend=biber,
    style=alphabetic,
    sorting=nyt,
    maxbibnames=99,
    giveninits=true
  ]{biblatex}
\usepackage{tikz}
\usetikzlibrary{shapes.geometric}
\usepackage[final]{showlabels}

\usepackage{subcaption}

\usepackage{varioref}

\usepackage{datetime}

\usepackage{mathtools}

\usepackage[h]{esvect}

\usepackage{array}

\usepackage{listings}
\usepackage{microtype}

\usepackage{booktabs}

\usepackage{tikz-cd}
\usetikzlibrary{decorations.pathreplacing}
\usetikzlibrary{calc,tikzmark}

\makeatletter
\DeclareFontEncoding{LS1}{}{}
\DeclareFontSubstitution{LS1}{stix}{m}{n}
\DeclareMathAlphabet{\mathcal}{LS1}{stixscr}{m}{n}
\makeatother

\usepackage{parskip}
\theoremstyle{plain}
    \newtheorem{theorem}{Theorem}[section]
    \newtheorem{lemma}[theorem]{Lemma}
    \newtheorem{proposition}[theorem]{Proposition}
    \newtheorem{corollary}[theorem]{Corollary}

    \newtheorem{mainthm}{Theorem}

\theoremstyle{definition}
    \newtheorem{definition}[theorem]{Definition}
    
    \newtheorem{example}[theorem]{Example}
    \newtheorem{remark}[theorem]{Remark}

\numberwithin{figure}{section}
\numberwithin{table}{section}
\usepackage[all]{hypcap}

\counterwithout{equation}{section}

 \DeclareFontFamily{U}{manual}{}
 \DeclareFontShape{U}{manual}{m}{n}{ <->  manfnt }{}
 \newcommand{\manfntsymbol}[1]{%
    {\fontencoding{U}\fontfamily{manual}\selectfont\symbol{#1}}}

\makeatletter
   \@addtoreset{section}{part}
   \@addtoreset{footnote}{section}

    {\hspace*{\fill}$\lrcorner$\endgraf\endgroup\end{trivlist}}
\makeatother

\def\ZZ{\mathbb{Z}}

\def\SS{\mathbb{S}}

\def\PP{\mathbb{P}}

\def\L{\mathcal{L}}

\def\CC{\mathbb{C}}

\def\|{\Vert}

\usepackage{calligra,mathrsfs}
\def\shom{\mathscr{H}\text{\kern -3pt {\calligra\Large om}}\,}

\def\ord{\operatorname{ord}}

\def\dt4{\text{DT}_4}

\def\->{\rightarrow}
\def\=>{\Rightarrow}

\def\<{\langle}
\def\>{\rangle}

\DeclareMathOperator{\Spec}{Spec}

\newcommand{\X}{\mathcal{X}}

\newcommand{\id}{\mathrm{id}}
\newcommand{\LL}{\mathcal{L}}

\newcommand{\DD}{\mathbb{D}}
\renewcommand{\SS}{\mathbb{S}}
\DeclareMathOperator{\HF}{HF}

\DeclareMathOperator{\Bl}{Bl}
\newcommand{\pr}{\mathrm{pr}}

\usepackage{import}
\usepackage{xifthen}
\usepackage{pdfpages}
\usepackage{transparent}

\title{Contact loci of semihomogeneous singularities}

\author{Eduardo de Lorenzo Poza}

\author{Jiahui Huang}

\begin{document}

\begin{abstract}
    Two of the main open problems in the theory of contact loci are the arc-Floer conjecture, which states that the compactly supported cohomology of the restricted $m$-contact locus and the fixed-point Floer cohomology of the $m$-th iterate of the Milnor monodromy are isomorphic up to a shift; and the embedded Nash problem, which asks for a description of the irreducible components of the unrestricted $m$-contact locus in terms of an embedded resolution of singularities. In this paper we study the geometry of contact loci of semihomogeneous singularities, and use our results to give an affirmative answer to the arc-Floer conjecture (under some conditions on the dimension and the degree) and a complete solution to the embedded Nash problem.
\end{abstract}

\maketitle
\thispagestyle{empty}


\section{Introduction}

Given an integer $m \geq 1$, the $m$-contact locus $\X_m$ associated to a hypersurface singularity is the subscheme of the jet space formed by jets that have order of contact $m$ with the singularity.
Contact loci were first defined by Denef and Loeser \cite{DL-lefnum} in the context of motivic integration, where they appeared as the coefficients of the motivic zeta function, see \cite[\S 7.4]{chambertloir2018}.

While the geometry of contact loci is interesting in its own right because the spaces $\X_m$ may be seen as invariants of the singularity, its study is also related to a number of open problems in singularity theory, such as the monodromy conjecture, the arc-Floer conjecture, and the embedded Nash problem.
For example, contact loci of plane curves are well understood ---their connected components are smooth and the topology of each component is completely described \cite[Thm. 3.5]{arc-floer-curves}---, and this is a key ingredient in the proof of the arc-Floer conjecture for plane curves \cite[Thm. 1.1]{arc-floer-curves} and the resolution of the embedded Nash problem for plane curves \cite[Thm. 1.22]{irred-curve}.

In this paper, we study the geometry of contact loci associated to semihomogeneous singularities, see Definition~\ref{def:semihomogeneous}.
Our main result in this direction is a description of the graded pieces of the filtration induced by the order of the jets, see Proposition~\ref{prop:filtered-topology}.
Even though each graded piece is relatively simple to describe, the relationship between the different pieces is quite intricate.
In contrast to the case of plane curves, we show that contact loci of semihomogeneous singularities may have several irreducible components and the irreducible components could be singular, see Example~\ref{ex:counterexample}.
Nevertheless, we also find that the cohomology of each contact locus is the direct sum of the cohomologies of its irreducible components.

As corollaries to our study of the geometry of contact loci, we show that the arc-Floer conjecture holds for semihomogeneous singularities under some conditions on the degree, and we give a complete solution to the embedded Nash problem for this kind of singularities.
Let us now give a brief overview of both problems and the techniques that we use in their resolution.

\subsection{The arc-Floer conjecture}

The \emph{arc-Floer conjecture} is a conjectural relationship between two invariants of an isolated hypersurface singularity: one of symplectic origin ---Floer cohomology of Milnor monodromy iterates---, and one of algebraic origin ---compactly supported cohomology of the contact loci. 
Its origin goes back to an observation by Seidel about a result by Denef and Loeser that we will now recall.

Let $f: (\CC^n, 0) \to (\CC, 0)$ be a holomorphic function germ with an isolated singularity at the origin. 
A classical result by Milnor (see \cite[\S 4]{seade2019} for a complete list of references) says that for small $\varepsilon, \delta > 0$, the restriction $f|: f^{-1}(\partial \mathbb{D}_\delta) \cap \overline{\mathbb{B}}_\varepsilon \longrightarrow \partial \mathbb{D}_\delta$ is a locally trivial fibration, whose monodromy (well defined up to isotopy) we call the \emph{Milnor monodromy} and denote by $\phi$.
On the other hand, Denef and Loeser considered the \emph{(restricted) contact loci} $\mathcal{X}_m$ of $f$, which are defined for each integer $m \geq 1$ as the subscheme of the $m$-th jet scheme $\L_m(\CC^n)$ formed by jets through the origin that have contact order $m$ with $f$.

The aforementioned result of Denef and Loeser \cite[Thm. 1.1]{DL-lefnum} states that the compactly supported Euler characteristic of the $m$-contact locus $\chi_c(\mathcal{X}_m)$ is equal to the Lefschetz number of the $m$-th iterate of the monodromy $\Lambda(\phi^m)$. 
Seidel observed that this Lefschetz number is also an Euler characteristic, namely the Euler characteristic of the Floer cohomology groups $\HF^\bullet(\phi^m, +)$.
The arc-Floer conjecture predicts that the numerical equality $\chi_c(\X_m) = \Lambda(\phi^m)$ actually comes from an isomorphism between the underlying the cohomology groups \cite[Conj. 1.5]{coho-contact}, see \eqref{eq:arc-floer-conjecture}.

Aside from Denef and Loeser's result, there is further evidence in favor of the arc-Floer conjecture.
First, there are two spectral sequences with isomorphic first pages (up to a shift) that converge to either of the two cohomology groups.
They were introduced by McLean \cite[Thm. 1.2]{mclean} and Budur, Fernández de Bobadilla, L\^{e} and Nguyen \cite[Thm. 1.1]{coho-contact}, respectively.
Furthermore, it is known that the isomorphism predicted by the arc-Floer conjecture exists when $m$ is the multiplicity of $f$ \cite[Prop. 1.6]{coho-contact}.
Moreover, the conjecture has been shown to hold for all $m$ in the case of plane curves by de la Bodega and the first author \cite[Thm. 1.1]{arc-floer-curves}.

Here we show that the arc-Floer conjecture holds for semihomogeneous singularities of degree $d \geq 2$ in $n \geq 3$ variables, under some conditions on the degree (see Theorem~\ref{thm:arc-floer-semihomogeneous} for the proof and the sharpest conditions we are able to get, also Figure~\ref{fig:scatter-plot}).

\begin{mainthm}
    Let $f: (\CC^n, 0) \to (\CC,0)$ be a semihomogeneous germ of degree $d$ in $n \geq 3$ variables. If $2 \leq d < n/2$ or $d > 2n - 2$, then \eqref{eq:arc-floer-conjecture} holds.
\end{mainthm}

The proof of this theorem involves a careful study of the spectral sequences by McLean and Budur et al.
Since the relationship between the $E_1$ pages of both spectral sequences is already known, we need to study the differentials.
The two spectral sequences depend on a choice of embedded resolution of singularities. 
By considering a particularly simple $m$-separating resolution, we show in Theorem \ref{thm:degeneration-mclss} and Theorem \ref{thm:degeneration-clss} that both spectral sequences degenerate at the first page (in general for the cohomology of contact loci, and under some conditions on $d$ and $n$ in the case of the McLean spectral sequence).
Note that this $E_1$-degeneration is somewhat unexpected because of our results on the geometry of the contact locus --- we show that the irreducible components of $\X_m$ are not disjoint, but $E_1$-degeneration means that $H^\bullet_c(\X_m)$ is a direct sum of the cohomologies of the irreducible components.


\subsection{The embedded Nash problem}

The \emph{embedded Nash problem}, first proposed in \cite[Rmk. 2.8]{ein2004}, consists in determining the irreducible components of $\X_m$ in terms of an ($m$-separating) resolution of singularities, see \eqref{eq:embedded-nash-problem} for the precise statement.
It inherits its name from the classical Nash problem, which asks for an analogous description of the irreducible components of the subset of arcs that pass through the singular locus.

The embedded Nash problem has been solved for several families of singularities, such as 
toric invariant ideals given by a cone inside an affine toric variety \cite[Thm. 5.11]{ishii2004}, 
hyperplane arrangements \cite[\S 7]{contact-hyperplane}, \cite[\S 4]{irred-curve}, 
plane curves \cite[\S 6, \S 7]{irred-curve}, \cite[Thm. 3.5]{arc-floer-curves}, 
and curves in an arbitrary surface \cite[Thm. 3.4]{delabodega2025}.

As a corollary to our study of the geometry of contact loci, we give a complete solution to the embedded Nash problem for semihomogeneous singularities of degree $d$ in arbitrary dimension $n \geq 2$.
In fact, we not only describe the divisorial $m$-valuations that correspond to irreducible components of the $m$-contact locus, but also the dlt and essential $m$-valuations, see Definition~\ref{def:m-valuations}. These are the divisorial valuations that appear in \cite[Thm. 1.13]{irred-curve}, an analog of de Fernex and Docampo's result for the Nash problem \cite[Thm. 1.1]{defernex2016}.

\begin{mainthm}
    Let $f: (\CC^n, 0) \to (\CC,0)$ be an semihomogeneous germ of degree $d$ in $n \geq 2$ variables. 
    Then we have the following solution to \eqref{eq:embedded-nash-problem}:
    \begin{enumerate}[label=(\roman*)]
        \item If $d < n$, then the triple $(\mathbb{C}^n, f^{-1}(0), 0)$ has no dlt $m$-valuations; it has no contact valuations if $m < d$, and exactly one contact $m$-valuation if $m \geq d$; and it has $\lfloor m/d \rfloor$ essential $m$-valuations.
        
        \item If $d \geq n$, then the dlt, contact and essential $m$-valuations of the triple $(\mathbb{C}^n, f^{-1}(0), 0)$ coincide, and there are exactly $\lfloor m/d \rfloor$ of them.
    \end{enumerate}
\end{mainthm}

\begin{proof}
    See Proposition~\ref{prop:contact-m-valuations}, Proposition~\ref{prop:dlt-m-valuations} and Corollary~\ref{cor:essential-m-valuations}.
\end{proof}

\section*{Acknowledgements}
We would like to thank Nero Budur, Javier Fernández de Bobadilla and Matthew Satriano for independently suggesting and supervising this project, as well as providing helpful ideas and encouragements. 
We further wish to express gratitude towards Adrian Dawid and Devlin Mallory for the helpful discussions.
The first author was supported by 1187425N from FWO, Research Foundation, Flanders.

\section{Floer cohomology of monodromy iterates}
\label{sec:floer}

Let $f: (\CC^n, 0) \to (\CC, 0)$ be a holomorphic function germ with an isolated critical point at $0 \in \CC^n$.
The goal of this section is to study the fixed-point Floer cohomology groups of monodromy iterates $\HF^\bullet (\phi^m, +)$ in the case where $f$ is a homogeneous polynomial of degree $d$.

\subsection{Preliminaries}

A \emph{log resolution} of a hypersurface singularity germ $f: (\CC^n, 0) \to (\CC, 0)$ is a proper birational morphism $\mu: (X, (f \circ \mu)^{-1}(0)) \to (\CC^n, f^{-1}(0))$ from a smooth variety $X$ such that the total transform $(f \circ \mu)^{-1}(0)$ is a divisor with simple normal crossings. 
We denote by $\mathcal{E}$ the set of irreducible components of the total transform, so $(f \circ \mu)^{-1}(0) = \sum_{E \in \mathcal{E}} N_E E$, where each number $N_E \coloneqq \ord_{E} f$ is called the \emph{multiplicity} of $f \circ \mu$ along the irreducible divisor $E$.
We denote by $\mathcal{P} \subseteq \mathcal{E}$ the subset of irreducible components of the strict transform of $(f \circ \mu)^{-1}(0)$ under $\mu$. 
We also define the \emph{log discrepancy} of each $E \in \mathcal{E}$ by $\nu_E \coloneqq \ord_{E} (K_{X/\CC^n}) + 1$. 

Given an integer $m \geq 1$, the divisors $E$ such that $N_E$ divides $m$ are called \emph{$m$-divisors}. 
If, whenever $E,F\in \mathcal{E}$ and $E \cap F \neq \varnothing$, we have $N_E + N_F > m$, then the resolution $\mu$ is said to be \emph{$m$-separating}.
It is always possible to perform further blow-ups on a given log resolution to make it $m$-separating \cite[Lemma 2.9]{coho-contact}.

For every divisor $E \in \mathcal{E}$ we denote $E^\circ \coloneqq E \setminus (\cup_{F \in \mathcal{E} \setminus \{E\}} F)$. 
In \cite[\S 2.3]{DL-lefnum}, Denef and Loeser constructed the following degree $N_E$ unramified covering of $E^\circ$. 
For any affine open $U \subset X$ such that $f \circ \mu = uv^{N_E}$, with $u \in \mathscr{O}^\times(U)$ and $v \in \mathscr{O}(U)$, we consider the pullback 
\begin{equation} \label{eq:pullback-covering}
\begin{tikzcd} 
	{\widetilde{E^\circ \cap U}} & {E^\circ \cap U} \\
	{\mathbb{C}^\times} & {\mathbb{C}^\times}
	\arrow["c_{E, u}", from=1-1, to=1-2]
	\arrow[from=1-1, to=2-1]
	\arrow["\Big\lrcorner"{anchor=center, pos=0.125}, draw=none, from=1-1, to=2-2]
	\arrow["u", from=1-2, to=2-2]
	\arrow["{z \mapsto z^{N_E}}", from=2-1, to=2-2]
\end{tikzcd}
\end{equation}
If on a different affine open $U' \subset X$ we have an expression of the same form $f \circ \mu = u'(v')^{N_E}$, with $u' \in \mathscr{O}^\times(U')$ and $v' \in \mathscr{O}(U')$, then in the intersection $U \cap U'$ we have $u = u'(v'/v)^{N_E}$ with $v'/v \in \mathscr{O}^\times(U \cap U')$, so by the proof of \cite[Lemma 1.4.1]{DL-motivic} there is a canonical isomorphism
\[\begin{tikzcd}[column sep=tiny, row sep=small]
	{c_{E,u}^{-1}(E^\circ \cap U \cap U')} && {c_{E,u'}^{-1}(E^\circ \cap U \cap U')} \\
	& {E^\circ \cap U \cap U'}
	\arrow["\cong", from=1-1, to=1-3]
	\arrow["{c_{E,u}}"', from=1-1, to=2-2]
	\arrow["{c_{E,u'}}", from=1-3, to=2-2]
\end{tikzcd}\]
Taking an affine open cover of $E^\circ$ and gluing via these canonical isomorphisms, we get a covering $c_E: \widetilde{E}^\circ \to E^\circ$ that we call the \emph{Denef-Loeser covering of $E^\circ$}. 
McLean has given a different description of this covering in terms of the fundamental group, see \cite[Thm. 1.2]{mclean}.

One of the most important objects associated to the hypersurface singularity $f$ is its \emph{Milnor fibration}. 
Milnor \cite[Ch. 5]{milnor1968} showed that there exist $0 < \delta \leq \varepsilon \leq 1$ such that the restriction
\[
f|: f^{-1}(\partial \mathbb{D}_\delta) \cap \overline{\mathbb{B}}_\varepsilon \longrightarrow \partial \mathbb{D}_\delta
\]
is a $C^\infty$-locally trivial fibration whose isotopy type does not depend on $\delta, \varepsilon$. 
Moreover, this fibration is a Liouville fibration when endowed with the restriction of the standard Liouville form $\lambda_\mathrm{std} \in \CC^n$.
The fiber $M = f^{-1}(\delta) \cap \overline{\mathbb{B}}_\varepsilon$ is called the \emph{Milnor fiber} of $f$, and the fibration admits a compactly supported monodromy $\phi: M \to M$ that is exact, i.e. such that the 1-form $\phi^* \lambda_\mathrm{std} - \lambda_\mathrm{std}$ is exact.

In this setting, it is possible to define the \emph{fixed-point Floer cohomology groups} $\HF^\bullet (\phi^m, +)$ associated to each monodromy iterate $\phi^m$, see \cite{dostoglou-salamon94}, \cite{seidel}, \cite{mclean},  \cite{Uljarevic}.
We do not recall their definition here, since we will not need it.
Instead, we recall the spectral sequence \eqref{eq:mclss} that we will use to compute them.
We remark that this spectral sequence has been put to use with great success in the last years.
For example McLean has used it to recover classical invariants of the singularity, such as the multiplicity or the log canonical threshold \cite[Cor. 1.4]{mclean}; and Fernández de Bobadilla and Pe{\l}ka have used it in their proof of the Zariski multiplicity conjecture for families \cite[Thm. 1.1]{bobadilla2024}.

\begin{theorem}[McLean spectral sequence, {\cite[Thm. 1.2]{mclean}}] \label{thm:mclss}
    Let $\mu:(X,(f\circ \mu)^{-1}(0))\->(\CC^n,f^{-1}(0))$ be an $m$-separating resolution. Fix an ample divisor $H = \sum_{E \in \mathcal{E}} b_E E$ on $X$ such that $b_E \in \mathbb{Z}_{\leq 0}$ for all $E \in \mathcal{E}$. 
    Then there is a spectral sequence $\{\phantom{}_\text{McLean} E_\ell^{p,q}\}_{\ell \geq 0}$ converging to Floer cohomology $\HF^\bullet(\phi^m, +)$ whose first page is
    \begin{equation} \label{eq:mclss}
        \phantom{}_\text{McLean} E_1^{p,q} = \bigoplus_{\substack{E \in \mathcal{E} \setminus  \mathcal{P}, \ N_E \mid m \\ mb_E/N_E = p}} H_{n-1-(p+q)-2m((\nu_E/N_E) - 1)}(\widetilde{E}^\circ; \mathbb{Z}).
    \end{equation}
\end{theorem}

\begin{remark} \label{rmk:index-fix}
    The degree of the homology groups in the first page of \eqref{eq:mclss} differ from the ones in the reference \cite[Thm. 1.2]{mclean} by a shift of $2m$ because of a known mistake in the computation \cite[Thm. 5.41(3)]{mclean}, see \cite[Rmk. 7.6]{bobadilla2024}.
\end{remark}

\subsection{An $m$-separating log resolution for a semihomogeneous singularity}
\label{subsec:m-sep-resolution}

Let us recall the definition of a semihomogeneous germ.

\begin{definition}[{cf. \cite[Ch. 12]{arnold1985}}] \label{def:semihomogeneous}
    A power series $f \in \mathbb{C} \llbracket x_1, \ldots, x_n \rrbracket$ is said to be \emph{semihomogeneous of degree $d$} if it is of the form $f = h + F$, where $h$ is a homogeneous polynomial of degree $d$ with $0$ as an isolated critical point and $F$ is a power series of order strictly greater than $d$.
\end{definition}

From now on we assume that $f = h + F$ is a semihomogeneous power series of degree $d$. 
Equivalently, the initial term $h$ defines a smooth projective hypersurface $S \subset \mathbb{P}^{n-1}$ of degree $d$. 
We assume $n \geq 3$ so that $S$ is connected; this is not a problem, since the Floer cohomology groups have already been computed for $n=2$ in full generality, see \cite[Cor. 5.5]{arc-floer-curves}.

Our first goal is to compute the first page of the McLean spectral sequence \eqref{eq:mclss}. 
To do so, we need to fix an $m$-separating resolution of singularities and to understand the topology of the unramified cyclic covers $\widetilde{E}^\circ$.
Note that the blow-up of $\mathbb{C}^n$ at the origin $\beta: \Bl_{0} \mathbb{C}^n \to \mathbb{C}^n$ is already a log resolution of the germ $f^{-1}(0)$.
Iteratively blowing up the intersections $E \cap F$ for $E, F$ exceptional divisors or strict transforms with $N_E + N_F \leq m$, we obtain a sequence of blow-ups 
\begin{equation} \label{eq:sequence-blow-ups}
    X = X^{(\ell)} \xrightarrow{\mu^{(\ell)}} X^{(\ell-1)} \xrightarrow{\mu^{(\ell-1)}} \cdots \xrightarrow{\mu^{(2)}} X^{(1)} = \Bl_{0} \mathbb{C}^n \xrightarrow{\mu^{(1)} = \beta} \mathbb{C}^n,
\end{equation}
whose composition $\mu: X \to \mathbb{C}^n$ is an $m$-separating log resolution. 
In fact, we will see in Proposition~\ref{prop:minimal-m-separating-log-res} that this is the minimal $m$-separating log resolution of the triple $(\mathbb{C}^n, f^{-1}(0), 0)$. 
The dual graph of $\mu$ is a chain (meaning that it is linearly ordered). 
The strict transform of $f^{-1}(0)$ (which is irreducible because we are assuming $n \geq 3$, and we denote it by $E_{(1,0)}$) lies on one end of the chain.
The strict transform of the exceptional divisor of the first blow-up $\beta$ (which we denote by $E_{(0,1)}$) lies on the other end of the chain, see Figure \ref{fig:resolution}.

\begin{figure}[ht]
    \centering
    \makebox[\textwidth][c]{\resizebox{0.7\linewidth}{!}{%
    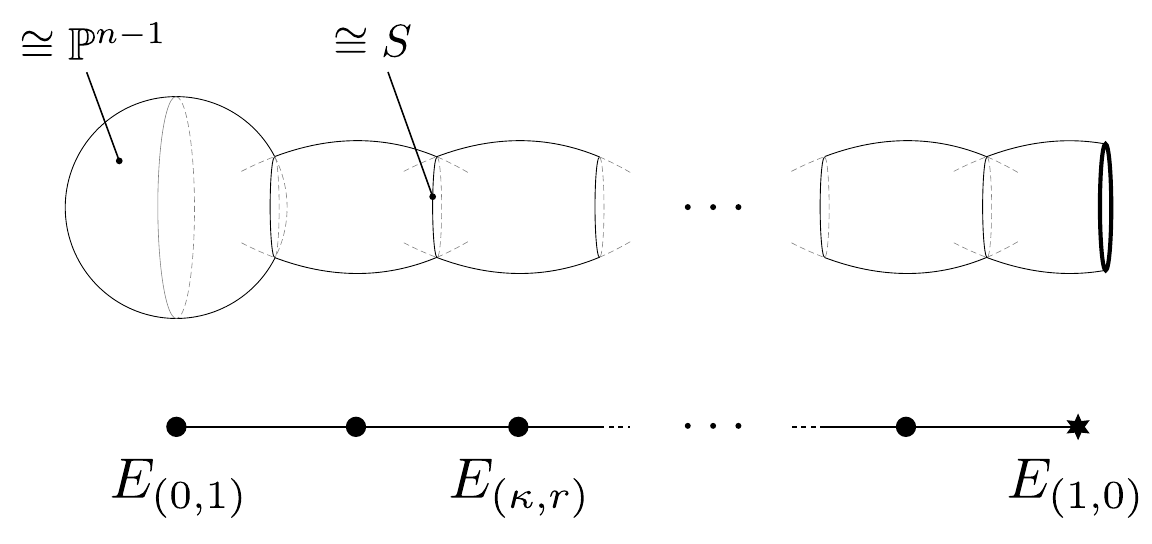
}}
    \caption{Sketch of the total transform $(f \circ \mu)^{-1}(0)$ of the $m$-separating log resolution $\mu: X \to \mathbb{C}^n$ and its dual graph.}
    \label{fig:resolution}
\end{figure}

Starting with the labels $E_{(0,1)}$ and $E_{(1,0)}$ that we have just defined, we inductively define $E_{(\kappa + \kappa', r + r)}$ to be the exceptional divisor obtained by blowing up the intersection of $E_{(\kappa, r)}$ and $E_{(\kappa', r')}$.
This way, we obtain a bijection between the exceptional divisors that may appear in a log resolution of the form \eqref{eq:sequence-blow-ups} and pairs of coprime non-negative integers $(\kappa, r)$, see \cite[Prop. 2.5]{arc-floer-curves}. 
If we denote by $N_{(\kappa, r)}$ and $\nu_{(\kappa, r)}$ the multiplicity and log discrepancy respectively of the divisor $E_{(\kappa, r)}$, we have 
\[
    N_{(1, 0)} = 1, \ \nu_{(1, 0)} = 1, \qquad
    N_{(0, 1)} = d, \ \nu_{(0, 1)} = n,
\]
see \cite[22.4.R]{vakil}, \cite[Prop. II.8.20]{hartshorne} for the computation of the log discrepancy. 
It follows by induction that for every coprime pair of non-negative integers $(\kappa, r)$ we have
\begin{equation} \label{eq:N-nu-formulas}
    N_{(\kappa, r)} = \kappa + rd, \qquad \nu_{(\kappa, r)} = \kappa + rn.
\end{equation}

\subsection{The topology of the cyclic covers}

Now that we have fixed an $m$-separating resolution, we need to compute the homology groups of the Denef-Loeser coverings $c_{(\kappa, r)}: \widetilde{E}^\circ_{(\kappa, r)} \to E_{(\kappa, r)}^\circ$, which we will finally achieve in Corollary~\ref{cor:homology-computation}.
Before we are able to do so, we need to gain some understanding of the topology of the exceptional divisors.

We start with $E_{(0,1)}$, which recall is the exceptional divisor of the blow-up $\beta: \Bl_0 \mathbb{C}^n \to \mathbb{C}^n$. 
Clearly $E_{(0,1)}$ is isomorphic to $\mathbb{P}^{n-1}$, and the intersection of $E_{(0,1)}$ with the next divisor in the chain of Figure \ref{fig:resolution} is sent to the smooth hypersurface $S$ under this isomorphism. 
Hence $\widetilde{E}_{(0,1)}^\circ$ is a degree $d$ cover of $E^\circ_{(0,1)} \cong \mathbb{P}^{n-1} \setminus S$.

\begin{proposition} \label{E01-topology}
    The space $\widetilde{E}_{(0,1)}^\circ$ is homeomorphic to $M_h \coloneqq \{x \in \mathbb{C}^n \mid h(x) = 1\}$, which is the affine Milnor fiber of $h$, the initial term of $f$, see \cite[Def. 3.1.12]{dimca1992}.
\end{proposition}

\begin{proof}
    For any point $(a_1 : \cdots : a_n) \in \mathbb{P}^{n-1}$ consider the ray through the origin in $\mathbb{C}^n$ parametrized by $t \in \mathbb{C} \mapsto (a_1 t, \ldots, a_n t) \in \mathbb{C}^n$. 
    Note that $(a_1 : \cdots : a_n) \in S$ if and only if $h(a_1 t, \ldots, a_n t) = 0$ for all $t \in \CC$.
    On the other hand, for every $(a_1 : \cdots : a_n) \in \mathbb{P}^{n-1} \setminus S$ the equation $t^d h(a_1, \ldots, a_n) = h(a_1 t, \ldots, a_n t) = 1$ has precisely $d$ solutions in $t$.
    This defines a degree $d$ covering $M_h \to \mathbb{P}^{n-1} \setminus S$ that is isomorphic to the covering $c_{(0,1)}: \widetilde{E}_{(0,1)}^\circ \to E^\circ_{(0,1)} \cong \mathbb{P}^{n-1} \setminus S$.
\end{proof}

Now we move on to the rest of the divisors.
For every exceptional divisor $E$ of $\mu$, let $E^{(i)} \subset X^{(i)}$ be the image of $E$ under the composition $\mu^{(i+1)} \circ \cdots \circ \mu^{(\ell)}$, see \eqref{eq:sequence-blow-ups}.
Note that for every $(\kappa, r) \neq (0, 1), (1, 0)$ we have $E_{(\kappa,r)}^{(1)} = E_{(0,1)}^{(1)} \cap E_{(1,0)}^{(1)} \cong S$, and hence we may consider the map
\[
    \pi = \pi_{(\kappa, r)} \coloneqq p \circ \mu^{(2)} \circ \cdots \circ \mu^{(\ell)}|: E_{(\kappa, r)} \to S
\]
where $p$ is the natural projection $\Bl_0 \mathbb{C}^n \to \mathbb{P}^{n-1}$.

\begin{proposition} \label{prop:exceptional-divisor-topology}
    For any $(\kappa, r) \neq (1, 0), (0, 1)$, the map $\pi: E_{(\kappa, r)} \to S$ is a $\mathbb{P}^1$-bundle with two distinguished sections, corresponding to the intersection of $E_{(\kappa, r)}$ with the adjacent divisors. 
    Removing either one of the sections results in a line bundle with Euler class $e(\pi_{(\kappa, r)}|) = \pm(\kappa + rd)h$, where $h \in H^2(S)$ is the restriction of the hyperplane class in $H^2(\mathbb{P}^{n-1})$ (the sign depends on which of the two sections we removed).
\end{proposition}

\begin{proof}

    For this proof we identify each divisor $E_{(\kappa, r)} \subset X$ with its blow-downs $E_{(\kappa, r)}^{(i)}$ to simplify the notation. 
    Suppose that $E_{(\kappa, r)}$ is the exceptional divisor obtained by blowing up the intersection of $E_{(\kappa', r')}$ and $E_{(\kappa'', r'')}$ in some intermediate step $X^{(i)}$ of the resolution. 
    Denote the intersection by $Z \coloneqq E_{(\kappa', r')} \cap E_{(\kappa'', r'')}$ and note that $Z \cong S$ has codimension 2 in $X^{(i)}$. Since $E_{(\kappa', r')}$ and $E_{(\kappa'', r'')}$ intersect transversally, the normal bundle of their intersection is the direct sum
    \[
        \mathcal{N}_{Z/X^{(i)}} \cong \mathcal{N}_{Z/E_{(\kappa', r')}} \oplus \mathcal{N}_{Z/E_{(\kappa'', r'')}}.
    \]
    Without loss of generality, suppose that $E_{(\kappa', r')}$ is closer to $E_{(0,1)}$ than $E_{(\kappa'', r'')}$ (i.e. further left in Figure \ref{fig:resolution}). Then we claim that
    \begin{equation} \label{eq:normal-bundle-computation}
        \mathcal{N}_{Z/E_{(\kappa', r')}} \cong \mathscr{O}_S(\kappa' + r'd) 
        \quad \text{ and } \quad 
        \mathcal{N}_{Z/E_{(\kappa'', r'')}} \cong \mathscr{O}_S(-\kappa'' - r''d).
    \end{equation}

    We prove \eqref{eq:normal-bundle-computation} by induction on the number of blow-ups. The base case is $(\kappa', r') = (0,1)$ and $(\kappa'', r'') = (1,0)$.
    Note that the tangent cone of $f^{-1}(0)$ is given by the vanishing of $h$, so it is the cone over the smooth projective hypersurface $S$. Therefore, after blowing up the origin, the normal bundle of the exceptional divisor inside the strict transform is $\mathcal{N}_{Z/E_{(1,0)}} \cong \mathscr{O}_S(-1)$.
    On the other hand, $E_{(0,1)}$ is isomorphic to $\mathbb{P}^{n-1}$, and under this isomorphism the inclusion $Z \hookrightarrow E_{(0,1)}$ is mapped to the inclusion $S \hookrightarrow \mathbb{P}^{n-1}$.
    Hence its normal bundle is $\mathcal{N}_{Z/E_{(0,1)}} \cong \mathscr{O}_S(d)$.
    This finishes the proof of the base case.

    Now suppose that \eqref{eq:normal-bundle-computation} holds.
    Recall that the exceptional divisor $E_{(\kappa, r)}$ of the blow-up
    \[
        \mu^{(i+1)}: X^{(i+1)} = \mathrm{Bl}_Z X^{(i)} \to X^{(i)}
    \]
    is isomorphic to the projectivization of the normal bundle, i.e $E_{(\kappa, r)} \cong \mathbb{P}(\mathcal{N}_{Z/X^{(i)}})$, see \cite[22.3.D.(a)]{vakil}, \cite[B.6.3]{fulton-intersection}.
    The normal bundle $\mathcal{N}_{Z/X^{(i)}}$ has two canonical rank-one subbundles, namely $\mathcal{N}_{Z/E_{(\kappa', r')}} \oplus 0$ and $0 \oplus \mathcal{N}_{Z/E_{(\kappa'', r'')}}$.
    By definition of projectivization of a vector bundle, these give two sections of $\mathbb{P}(\mathcal{N}_{Z/X^{(i)}}) \to Z$. 
    The images of these sections are sent under the isomorphism $\mathbb{P}(\mathcal{N}_{Z/X^{(i)}}) \cong E_{(\kappa, r)}$ to the intersections $E_{(\kappa, r)} \cap E_{(\kappa', r')}$ and $E_{(\kappa, r)} \cap E_{(\kappa'', r'')}$ respectively.

    Recall that tensoring by a line bundle does not affect the projectivization \cite[Lemma II.7.9]{hartshorne}, so tensoring with $\mathscr{O}_S(-\kappa' - r'd)$ and recalling that $(\kappa,r) = (\kappa', r') + (\kappa'', r'')$, we get
    \[
        \mathbb{P}(\mathscr{O}_S(\kappa' + r'd) \oplus \mathscr{O}_S(-\kappa'' - r''d)) \cong \mathbb{P}(\mathscr{O}_S \oplus \mathscr{O}_S(-\kappa - rd)).
    \]
    The subspace $0 \oplus \mathscr{O}_S(-\kappa'' - r''d)$ is sent by the tensor product to $0 \oplus \mathscr{O}_S(-\kappa - rd)$, which is the hyperplane at infinity of the projective completion $\mathbb{P}(\mathscr{O}_S \oplus \mathscr{O}_S(-\kappa - rd))$ of $\mathscr{O}_S(-\kappa - rd)$, see \cite[B.5.2]{fulton-intersection}. Removing this hyperplane at infinity we are left with the line bundle $\mathscr{O}_S(-\kappa - rd)$, whose zero section now corresponds to the intersection of $E_{(\kappa, r)}$ with $E_{(\kappa', r')}$. In conclusion, the normal bundle of the intersection $E_{(\kappa', r')} \cap E_{(\kappa, r)}$ in $E_{(\kappa, r)}$ is
    \[
        \mathcal{N}_{E_{(\kappa', r')} \cap E_{(\kappa, r)} / E_{(\kappa, r)}} = \mathscr{O}_S(-\kappa - rd).
    \]
    Analogously, we find
    \[
        \mathcal{N}_{E_{(\kappa'', r'')} \cap E_{(\kappa, r)} / E_{(\kappa, r)}} = \mathscr{O}_S(\kappa + rd).
    \]
    This concludes the induction, proving \eqref{eq:normal-bundle-computation} and finishing the proof of (ii). 
\end{proof}

Now that we have understood the topology of the divisors $E_{(\kappa, r)}$, we move on to the study of the Denef-Loeser coverings $c_{(\kappa, r)}: \widetilde{E}_{(\kappa, r)}^\circ \to E_{(\kappa, r)}^\circ$. 
The proof of the following lemma is standard, but we include it for completeness.

\begin{lemma} \label{Etilde-path-connected}
    For $(\kappa, r) \neq (0,1), (1,0)$, the space $\widetilde{E}_{(\kappa, r)}^\circ$ is path connected.
\end{lemma}

\begin{proof}
    Write $E = E_{(\kappa, r)}$ and let $E' \in \mathcal{E}$ be such that $E \cap E' \neq \varnothing$. 
    Then either $E$ has appeared during the resolution as the exceptional divisor of the blow-up of the intersection of $E'$ with another divisor, or vice versa.
    Denoting by $N, N'$ the multiplicities of $E, E'$ respectively, a straightforward induction shows that $\gcd(N, N') = \gcd(N_{(1,0)}, N_{(0,1)}) = \gcd(1,d) = 1$.

    Consider an open affine $U \subset X$ around a point in $E \cap E'$ such that $f \circ \mu = uv^{N}(v')^{N'}$, with $u \in \mathscr{O}^\times(U)$ and $v, v' \in \mathscr{O}(U)$, which exists because the total transform $(f \circ \mu)^{-1}(0)$ has normal crossings. 
    Note that $U \setminus E'$ is an affine open containing $E^\circ \cap U$. Furthermore, $v' \in \mathscr{O}^\times(U \setminus E')$ and hence $u(v')^{N'} \in \mathscr{O}^\times(U \setminus E')$. 
    By construction of the covering $c_E$ we have a description of $c_E^{-1}(E^\circ \cap U)$ as a pullback \eqref{eq:pullback-covering}, namely
    \[
        c_E^{-1}(E^\circ \cap U) \cong \{ (z, x) \in \mathbb{C}^\times \times (E^\circ \cap U) \mid u(x) v'(x)^{N'} = z^N \} .
    \]
    Consider a loop $x(t)$ in $E^\circ \cap U$ such that $v'(x(t)) = e^{2 \pi \imath t}$.
    Note that $u(x(t))$ is a loop in $\mathbb{C}^\times$ which does \emph{not} wind around the origin (because $u$ does not vanish along $E'$).
    Therefore we may lift $u(x(t))$ through the covering $\mathbb{C}^\times \to \mathbb{C}^\times: z \mapsto z^N$, obtaining a \emph{loop} $\hat{u}(t)$ in $\mathbb{C}^\times$ such that $\hat{u}(0) = \hat{u}(1)$ and $\hat{u}(t)^N = u(x(t))$.
    Define
    \[
        z(t) \coloneqq \hat{u}(t) e^{2 \pi \imath t N'/N}.
    \]
    Then,
    \[
        z(t)^N = \hat{u}(t)^N e^{2 \pi \imath t N'} = u(x(t)) v'(x(t))^{N'},
    \]
    which shows that $(z(t), x(t))$ is a loop in $c_E^{-1}(E^\circ \cap U)$ lifting $x(t)$. 
    Furthermore, as $t$ varies through the integers, $z(t)$ runs through all $N$-th roots of $u(x(0))$ because $N$ and $N'$ are coprime. 
    In other words, the path $(z(t), x(t))$ passes through all of the $N$ preimages of $u(x(0))$ under the covering $c_E$, which shows that $c_E^{-1}(E^\circ \cap U)$ is path connected. 
    Since any point in $\widetilde{E}^\circ$ may be connected to a point in $c_E^{-1}(E^\circ \cap U)$ by a path, we conclude that $\widetilde{E}^\circ$ is path connected.
\end{proof}

\begin{proposition} \label{prop:euler-class-covering}
    Let $(\kappa, r) \neq (0, 1), (1, 0)$.
    The composition $\widetilde{\pi}_{(\kappa, r)} \coloneqq \pi_{(\kappa, r)} \circ c_{(\kappa, r)}: \widetilde{E}_{(\kappa, r)}^\circ \to S$ is a fiber bundle with fiber $\mathbb{C}^\times$, and its Euler class is $e(\pi \circ p) = \pm h$, depending on the orientation.
\end{proposition}

\begin{proof}
    Recall from Proposition~\ref{prop:exceptional-divisor-topology} that $\pi_{(\kappa, r)}$ is a $\mathbb{P}^1$-bundle, and removing the two sections corresponding to the intersections of $E_{(\kappa, r)}$ with the adjacent divisors we obtain a $\CC^\times$-bundle.
    The composition $\widetilde{\pi}_{(\kappa,r)}$ of the covering map $c_{(\kappa, r)}$ followed by the $\CC^\times$-bundle $\pi_{(\kappa,r)}$ is a fiber bundle, see \cite{MO-fiberbundlecomposition}. 
    Let $F = \mathbb{C}^\times$ be the fiber of $\pi_{(\kappa, r)}$ and let $\widetilde{F}$ be the fiber of $\widetilde{\pi}_{(\kappa, r)}$.
    The long exact sequence in homotopy associated to the fibration $\widetilde{\pi}_{(\kappa, r)}$ reads
    \[
        \cdots \to \pi_1(S) \to \pi_0(\widetilde{F}) \to \pi_0(\widetilde{E}_{(\kappa, r)}^\circ) \to \cdots
    \]
    Recall that $\pi_1(S) = 0$ by the Lefschetz hyperplane theorem, and $\pi_0(\widetilde{E}_{(\kappa, r)}^\circ) = 0$ by Lemma~\ref{Etilde-path-connected}.
    Hence $\pi_0(\widetilde{F}) = 0$, so $\widetilde{F}$ is path connected.
    Since $\widetilde{F} = c_{(\kappa, r)}^{-1}(F)$, the covering $c_{(\kappa, r)}$ restricts to an $N_{(\kappa, r)}$-covering $\widetilde{F} \to F$, and since $F \cong \mathbb{C}^\times$ and $\widetilde{F}$ is path connected, we conclude $\widetilde{F} \cong \mathbb{C}^\times$.

    Now consider the map of fibrations
    \[\begin{tikzcd}
        {\widetilde{F}} & {\widetilde{E}^\circ_{(\kappa,r)}} & S \\
        F & {E^\circ_{(\kappa,r)}} & S
        \arrow[from=1-1, to=1-2]
        \arrow["{c_{(\kappa, r)}|_{\widetilde{F}}}"', from=1-1, to=2-1]
        \arrow[from=1-2, to=1-3]
        \arrow["{c_{(\kappa, r)}}", from=1-2, to=2-2]
        \arrow["\id", from=1-3, to=2-3]
        \arrow[from=2-1, to=2-2]
        \arrow[from=2-2, to=2-3]
    \end{tikzcd}\]
    By functoriality of the long exact sequence in homotopy, we have a commutative diagram
    \[\begin{tikzcd}
        {\pi_2(S) \cong \mathbb{Z}} & {\pi_1(\widetilde{F}) \cong \mathbb{Z}} \\
        {\pi_2(S) \cong \mathbb{Z}} & {\pi_1(F) \cong \mathbb{Z}}
        \arrow["{\widetilde{\partial}}", from=1-1, to=1-2]
        \arrow["\id"', from=1-1, to=2-1]
        \arrow["{\times (\pm N_{(\kappa, r)})}", from=1-2, to=2-2]
        \arrow["\partial", from=2-1, to=2-2]
    \end{tikzcd}\]
    By obstruction theory, the connecting morphisms $\partial$ and $\widetilde{\partial}$ are given by multiplication by the Euler classes of the fiber bundles $\pi_{(\kappa, r)}$ and $\widetilde{\pi}_{(\kappa, r)}$ respectively.
    Indeed, the map $\partial$ is defined as follows: consider a class $[\alpha] \in \pi_2(S)$ represented by a map $\alpha:(\DD^2, \partial \DD^2) \to (S, s_0)$ (where $s_0$ is the basepoint of $S$).
    Interpreting $(\DD^2, \partial \DD^2) \cong ((\SS^1 \times [0,1])/(\SS^1 \times \{0\}), \SS^1 \times \{1\})$, we may use the homotopy lifting property of the fibration $\pi_{(\kappa, r)}$ to lift $\alpha$ to a map $\widetilde{\alpha}: \DD^2 \to E^\circ_{(\kappa, r)}$ such that $\alpha(\partial \DD^2) \subset \pi_{(\kappa, r)}^{-1}(s_0) = F$. 
    Then $\partial([\alpha])$ is by definition the class of the restricted map $\widetilde{\alpha}|: \partial \DD^2 \cong \SS^1 \to F$.
    This class $[\widetilde{\alpha}] \in \pi_1(F)$ is precisely the obstruction to lift $\alpha$ to a map $\SS^2 \to E^\circ_{(\kappa, r)}$, and this obstruction is known to equal the Euler class.
    For the details, see \cite[\S 12]{milnor1974}.
    
    We know that the Euler class of $\pi_{(\kappa, r)}$ is $\pm(\kappa + rd)h$ by Proposition \ref{prop:exceptional-divisor-topology}, so the bottom arrow is multiplication by $\pm(\kappa + rd)$.
    On the other hand, $N_{(\kappa, r)} = \kappa + rd$ by \eqref{eq:N-nu-formulas}, so the right arrow is also multiplication by $\pm(\kappa + rd)$.
    Therefore the commutativity of the diagram gives that the top arrow $\widetilde{\partial}$ is multiplication by $\pm 1$, which implies that the Euler class of $\widetilde{\pi}_{(\kappa, r)}$ is $\pm h$, as we wanted.
\end{proof}

Now that we have described the topology of the Denef-Loeser coverings, we can compute their homology groups.
For the next results, we need to recall the structure of the cohomology ring of the smooth degree $d$ hypersurface $S$, see \cite{qiaochuyuan}. We have
\[
    H^k(S) \cong \begin{cases}
        \mathbb{Z} \langle h^{k/2} \rangle & \text{if } k < n-2 \text{ and $k$ is even}, \\
        \mathbb{Z}^b & \text{if } k = n - 2, \\
        \mathbb{Z} \langle \tfrac{1}{d} h^{k/2} \rangle & \text{if } k > n-2 \text{ and $k$ is even}, \\
        0 & \text{otherwise,}
    \end{cases}
\]
where $h \in H^2(S)$ is the hyperplane class coming from the embedding in $\mathbb{P}^{n-1}$, and $\tfrac{1}{d} h^{k/2}$ is just notation for an integral class $\beta \in H^k(S)$ such that $d \beta = h^{k/2}$. The rank of the middle cohomology group is given by 
\begin{equation} \label{eq:middle-cohomology-rank}
    b = (-1)^{n-1} 2 \left\lceil \frac{n-1}{2} \right\rceil + \sum_{k = 0}^{n-2} (-1)^k \binom{n}{k} d^{n-1-k}.
\end{equation}
From here it is immediate that the Lefschetz map $H^k(S) \xrightarrow{\smile h} H^{k+2}(S)$ is an isomorphism for $k \notin \{ n-4, n-3, n-2 \}$. 
Furthermore, if $n$ is odd then the maps $H^{n-4}(S) = 0 \xrightarrow{\smile h} H^{n-2}(S)$ and $H^{n-2}(S) \xrightarrow{\smile h} H^{n}(S) = 0$ are zero, and the map $H^{n-3}(S) \cong \mathbb{Z} \xrightarrow{\smile h} H^{n-1}(S) \cong \mathbb{Z}$ is multiplication by $d$.
On the other hand, if $n$ is even then $H^{n-3}(S) = 0 \xrightarrow{\smile h} H^{n-1}(S) = 0$ is zero, the map $H^{n-4}(S) \cong \mathbb{Z} \xrightarrow{\smile h} H^{n-2}(S)$ is injective (because the composition $H^{n-4}(S) \xrightarrow{\smile h^2} H^{n}(S)$ is multiplication by $d$ and hence injective), and $H^{n-2}(S) \xrightarrow{\smile h} H^{n}(S) \cong \mathbb{Z}$ is surjective by the following lemma.
    

\begin{lemma} \label{lem:lefschetz-surjective}
    The Lefschetz map $H^{n-2}(S) \xrightarrow{\smile h} H^n(S)$ is surjective.
\end{lemma}

\begin{proof}
    The following argument is due to \cite{MO-Lefschetz-map}. If $n$ is odd then $H^n(S) = 0$ and there is nothing to prove, so let us assume that $n$ is even. Since any two smooth hypersurfaces of degree $d$ are diffeomorphic, and the induced isomorphism in cohomology respects the cup product and the hyperplane class, we may assume that $S$ is the Fermat hypersurface in $\mathbb{P}^{n-1}$, i.e. the zero locus of $\sum_{i=1}^{n/2} (x_i^d - y_i^d)$. Then $S$ contains a linear subspace $L$ of dimension $\tfrac{n}{2} - 1$, namely the one defined by the ideal $(x_1 - y_1, \ldots, x_{n/2} - y_{n/2})$.

    Let $\alpha \in H^{n-2}(S)$ be the Poincaré dual of $[L] \in H_{n-2}(S)$. Then, $\alpha \smile h^{(n-2)/2}$ is the Poincaré dual of the intersection of $L$ with $(n-2)/2$ generic hyperplanes, which is a point, and therefore it generates $H^{2n-4}(S)$. Since $S$ has degree $d$, we have $h^{n-2} = d(\alpha \smile h^{(n/2) - 1})$. This implies that $\alpha \cup h$ is a generator of $H^n(S)$, which finishes the proof.
\end{proof}

\begin{corollary} \label{cor:homology-computation}
    The homology groups appearing in the McLean spectral sequence \eqref{eq:mclss} are
    \[
        H_k(\widetilde{E}_{(0, 1)}^\circ) \cong \begin{cases}
            \mathbb{Z} & \text{if } k = 0, \\
            \mathbb{Z}^{\mu(h)} & \text{if } k = n - 1, \\
            0 & \text{otherwise,}
        \end{cases}
    \]
    and for $(\kappa, r) \neq (1, 0), (0, 1)$ we have
    \begin{align*}
        \text{if } &n \text{ is odd:} & \text{if } &n \text{ is even:} \\
        H_k(\widetilde{E}_{(\kappa, r)}^\circ) &\cong \begin{cases}
            \mathbb{Z} & \text{if } k = 0, \\
            \mathbb{Z}^b \oplus \mathbb{Z}/(d) & \text{if } k = n-2, \\
            \mathbb{Z}^b & \text{if } k = n-1, \\
            \mathbb{Z} & \text{if } k = 2n-3, \\
            0 & \text{otherwise,}
        \end{cases}
        &
        H_k(\widetilde{E}_{(\kappa, r)}^\circ) &\cong \begin{cases}
            \mathbb{Z} & \text{if } k = 0, \\
            \mathbb{Z}^{b-1} & \text{if } k = n-2, \\
            \mathbb{Z}^{b-1} & \text{if } k = n-1, \\
            \mathbb{Z} & \text{if } k = 2n-3, \\
            0 & \text{otherwise,}
        \end{cases}
    \end{align*}
    where $\mu(h)$ is the Milnor number of $h$, and $b$ is the rank of the middle cohomology group of $S$, see \eqref{eq:middle-cohomology-rank}.
\end{corollary}

\begin{proof}
    
    By Proposition \ref{E01-topology} and \cite[Thm. 6.5]{milnor1968}, $\widetilde{E}^\circ_{(0,1)}$ has the homotopy type of a bouquet of $\mu(h)$ spheres of dimension $n-1$. This shows that its homology groups are as claimed.

    For all $(\kappa, r) \neq (1,0), (0,1)$, the space $\widetilde{E}^\circ_{(\kappa, r)}$ is a manifold of (real) dimension $(2n-2)$, so in order to know its homology we can instead compute its compactly supported cohomology and apply Poincaré duality.
    For this purpose, we are going to use the compactly supported Thom-Gysin sequence of the fibration $\widetilde{\pi}_{(\kappa, r)}: \widetilde{E}^\circ_{(\kappa, r)} \to S$, whose Euler class we know thanks to Proposition~\ref{prop:euler-class-covering}.
    Recall that the compactly supported Thom-Gysin sequence \cite[Thm. 5.7.11]{spanier1966} is
    \[
        \cdots \to H^k_c(\widetilde{E}^\circ_{(\kappa, r)}) \longrightarrow H^{k-2}_c(S) \xrightarrow{\smile e(\widetilde{\pi}_{(\kappa, r)})} H^{k}_c(S) \longrightarrow H^{k+1}_c(\widetilde{E}^\circ_{(\kappa, r)}) \to \cdots
    \]
    Since $S$ is compact, its compactly supported cohomology is just the usual cohomology that we recalled above.
    The parts of the Thom-Gysin sequence that are relevant to us are
    \[\begin{tikzcd}[row sep=0pt]
        {H^{-2}(S)} & {H^{0}(S)} & {H^{1}_c(\widetilde{E}^\circ_{(\kappa, r)})} & {H^{-1}(S)} \\
        0 & {\mathbb{Z}} && 0
        \arrow[from=1-1, to=1-2]
        \arrow[from=1-2, to=1-3]
        \arrow[from=1-3, to=1-4]
    \end{tikzcd}\]
    and also
    \[\begin{tikzcd}[row sep=0pt]
        {H^{2n-3}(S)} & {H^{2n-2}_c(\widetilde{E}^\circ_{(\kappa, r)})} & {H^{2n-4}(S)} & {H^{2n-2}(S)} \\
        0 && {\mathbb{Z}} & 0
        \arrow[from=1-1, to=1-2]
        \arrow[from=1-2, to=1-3]
        \arrow[from=1-3, to=1-4]
    \end{tikzcd}\]
    and finally,
    \[\begin{tikzcd}[row sep=0pt, column sep=4pt]
        {H^{n-4}(S)} & {H^{n-2}(S)} & {H^{n-1}_c(\widetilde{E}^\circ_{(\kappa,r)})} & {H^{n-3}(S)} & {H^{n-1}(S)} & {H^{n}_c(\widetilde{E}^\circ_{(\kappa,r)})} & {H^{n-2}(S)} & {H^{n}(S)} \\
        0 & {\mathbb{Z}^b} && {\mathbb{Z}} & {\mathbb{Z}} && {\mathbb{Z}^b} & 0 \\
        {\mathbb{Z}} & {\mathbb{Z}^b} && 0 & 0 && {\mathbb{Z}^b} & {\mathbb{Z}}
        \arrow[from=1-1, to=1-2]
        \arrow[from=1-2, to=1-3]
        \arrow[from=1-3, to=1-4]
        \arrow[from=1-4, to=1-5]
        \arrow[from=1-5, to=1-6]
        \arrow[from=1-6, to=1-7]
        \arrow[from=1-7, to=1-8]
        \arrow[hook, from=3-1, to=3-2]
        \arrow[two heads, from=3-7, to=3-8]
        \arrow["{\times d}", from=2-4, to=2-5]
    \end{tikzcd}\]
    where the first line indicates the value of each group in case $n$ is odd, and the second line is the case where $n$ is even. 
    
    In the case where $n$ is odd, we see that $H^{n-1}_c(\widetilde{E}^\circ_{(\kappa, r)}) \cong \mathbb{Z}^b$ and that $H^{n}_c(\widetilde{E}^\circ_{(\kappa, r)})$ is an extension of $\mathbb{Z}^b$ by $\mathbb{Z}/(d)$, which must be trivial because $\mathbb{Z}^b$ is free. 
    In the case where $n$ is even, the map $\mathbb{Z}^b \twoheadrightarrow \mathbb{Z}$ is surjective by Lemma~\ref{lem:lefschetz-surjective} and therefore the $H^{n}_c(\widetilde{E}^\circ_{(\kappa, r)}) \cong \mathbb{Z}^{b-1}$.
    On the other hand, the image of $\ZZ \hookrightarrow \ZZ^b$ is generated by $h^{(n-2)/2}$, which is a primitive element. 
    Indeed, suppose $h^{(n-2)/2} = k \beta$ for some $k \in \ZZ$, $\beta \in H^{n-2}(S)$ and consider the class $\alpha$ constructed in the proof of Lemma~\ref{lem:lefschetz-surjective}. Then,
    \[
    kd(\alpha \smile \beta) = d(\alpha \smile (k\beta)) = d(\alpha \smile h^{(n-2)/2}) = h^{n-2}
    \]
    which shows $k = \pm 1$ because $h^{n-2}$ is primitive. Hence the quotient is $H^{n-1}_c(\widetilde{E}^\circ_{(\kappa, r)}) \cong \ZZ^{b-1}$.
\end{proof}

\subsection{Degeneration of the McLean spectral sequence}

Now that we have computed the homology of the cyclic covers, we just need to place each homology group at the correct $p,q$ coordinates of the $E_1$ page. 
To do so, we need some information about the quotients $b_E/N_E$.

\begin{lemma} \label{b_E-monotone}
    There exists an ample divisor $H = \sum_{E \in \mathcal{E}} b_E E$ on $X$ with $b_E \leq 0$ for all $E \in \mathcal{E}$ and satisfying the following property: 
    if $E, F \in \mathcal{E}$ are such that $E$ is closer to $E_{(0,1)}$ than $F$ in the chain of divisors (i.e. further left in Figure \ref{fig:resolution}), then $b_E/N_E < b_F/N_F$.
\end{lemma}

\begin{proof}
    We construct the divisor $H$ inductively. 
    As the base case, consider the blow-up $\beta: \Bl_0 \mathbb{C}^n \to \mathbb{C}^n$, so that the only divisors in $\mathcal{E}$ are the strict transform $E_{(1,0)}$ and the exceptional divisor $E_{(0,1)}$. 
    Then $H = -E_{(0,1)}$ has the desired properties, since it is ample by \cite[Prop. II.7.10]{hartshorne} and $b_{(0,1)}/N_{(0,1)} = -1/d < 0/1 = b_{(1,0)}/N_{(1,0)}$.

    Now suppose that we have already constructed an ample divisor $H = \sum_{E \in \mathcal{E}} b_E E$ with the desired properties for a resolution $\mu: X \to \mathbb{C}^n$. 
    Suppose that we blow-up the intersection $E \cap F$ of two adjacent divisors in $\mathcal{E}$, obtaining a map $\mu': X' \coloneqq \Bl_{E \cap F} X \to X$.
    Denote the exceptional divisor of $\mu'$ by $E'$.
    Consider the divisor $H' = -E' + a \cdot (\mu')^*(H)$ on $X'$, where $a$ is a large positive integer that we still have to choose.
    By \cite[Prop. II.7.10]{hartshorne}, $H'$ is ample if $a$ is large enough. Furthermore, for large enough $a$ we also have
    \[
        \frac{ab_E}{N_E} < \frac{-1 + a(b_E + b_F)}{N_E + N_F} < \frac{ab_F}{N_F}, \text{ i.e. }
        \frac{b'_E}{N_E} < \frac{b'_{E'}}{N_{E'}} < \frac{b'_F}{N_F},
    \]
    as we wanted. This concludes the induction and the proof.
\end{proof}

    

\begin{theorem} \label{thm:degeneration-mclss}
    Let $n \geq 3, d \geq 1, m \geq 1$ be integers. Suppose that 
    \begin{equation} \label{eq:condition-degeneration}
        2k(d-n) + 1 \notin \{ \pm 1, \pm(n-1), \pm(n-2), \pm(2n-3) \} \quad \text{for all } k \in [1, m/d) \cap \ZZ. 
    \end{equation}
    Then the McLean spectral sequence \eqref{eq:mclss} associated to the minimal $m$-separating log resolution of a semihomogeneous singularity of degree $d$ in $\mathbb{C}^n$ degenerates at the first page. In particular, if $2 \leq d < n/2$ or $d > 2n - 2$ then the spectral sequence degenerates for all $m \geq 1$.
\end{theorem}

\begin{proof}
    By \eqref{eq:N-nu-formulas}, the $m$-divisors in the resolution $\mu: X \to \mathbb{C}^n$ are in one-to-one correspondence with coprime pairs $(\kappa, r) \in \mathbb{Z}_{\geq 0}^2$ such that $\kappa + rd$ divides $m$. 
    In turn, these are in one-to-one correspondence with (not necessarily coprime) pairs $(a, b) \in \mathbb{Z}_{\geq 0}^2$ such that $a + bd = m$ (the bijection is given by $(a, b) \mapsto (\kappa, r) = (a, b)/\gcd(a,b)$). 
    The pairs satisfying this latter condition are precisely
    \begin{equation} \label{eq:kappa-r-formula}
        \{ (a, b) = (m + id, -i) \mid i \in [-m/d, 0] \cap \ZZ \}.
    \end{equation}
    Let $E_{-\lfloor m/d \rfloor}, \ldots, E_{-1}, E_0$ be the corresponding $m$-divisors, so that $E_{i} = E_{(m+id, -i)/\gcd}$.
    Note that $E_0$ is the strict transform $E_{(1,0)}$, and whenever $d$ divides $m$ we have $E_{-m/d} = E_{(0,1)}$.
    For a cleaner notation, denote by $N_i, \nu_i, b_i$ the constants $N_{E_i}, \nu_{E_i}, b_{E_i}$ associated to $E_i$.
    The indexing of the $E_i$ has been chosen so that $i < j$ if and only if $E_i$ is closer to $E_{(0,1)}$ than $E_j$, and hence $b_i/N_i < b_j/N_j$ by Lemma \ref{b_E-monotone}.
    Since we know the coprime pair associated to each divisor $E_i$, we can compute explicitly using \eqref{eq:N-nu-formulas} that
    \begin{equation} \label{eq:nu/N-computation}
        2m\left(\frac{\nu_i}{N_i} - 1\right) = 2m \left(  \frac{(m+id) - in}{(m+id) - id} - 1 \right) = 2i(d-n).
    \end{equation}

    Let $p_i = mb_i/N_i$, so that the homology of $\widetilde{E}_i^\circ$ appears in the column $p_i$ of the first page of the McLean spectral sequence \eqref{eq:mclss}.
    Let $\ell \geq 1$ be an integer, and suppose that we have already shown that all the differentials of the McLean spectral sequence are zero up to the $\ell$-th page, so $E_{\ell}^{p,q} \cong E_1^{p,q}$ for all $p, q$.
    Suppose that there is a nonzero differential on the page $E_{\ell}$, which necessarily goes from a column $p_i$ to a column $p_j$ with $i < j$, say $0 \neq d_{\ell}^{p_i, p_j}: E_{\ell}^{p_i, q_i} \to E_{\ell}^{p_j, q_j}$.
    By definition of the differentials, we have 
    \begin{equation} \label{pj-qj-formulas}
        p_j = p_i + \ell \qquad \text{and} \qquad q_j = q_i - \ell + 1.
    \end{equation}
    Using Theorem \ref{thm:mclss} and \eqref{eq:nu/N-computation}, we have
    \[
        E_{\ell}^{p_i, q_i} \cong E_1^{p_i, q_i} = H_{n-1-(p_i+q_i)-2i(d-n)}(\widetilde{E}_i^\circ),
    \]
    and analogously for $E_{\ell}^{p_j, q_j}$. Since we are assuming that the differential is nonzero, both the domain and codomain must be nonzero. 
    By Corollary \ref{cor:homology-computation}, this implies $-\lfloor m/d \rfloor \leq i < j \leq -1$ and
    \[
    \begin{cases}
        n-1-(p_i+q_i)-2i(d-n) \in \{ 0, n-2, n-1, 2n-3 \}, \\
        n-1-(p_j+q_j)-2j(d-n) \in \{ 0, n-2, n-1, 2n-3 \}.
    \end{cases}   
    \]
    Note that this condition covers both cases $(\kappa, r) = (0, 1)$ and $(\kappa, r) \neq (0, 1)$. Subtracting both terms and using \eqref{pj-qj-formulas}, we get
    \[
        2(j-i)(d-n) + 1 \in \{ 0, \pm 1, \pm(n-1), \pm(n-2), \pm(2n-3) \}.
    \]
    To finish the proof, put $k = j - i$ and observe that $1 \leq k \leq \lfloor m/d \rfloor - 1$ and that $2k(d-n) + 1$ is an odd integer, and therefore is never zero.

    Now suppose $d > 2n - 2$. Then for $1 \leq k \leq \lfloor m/d \rfloor - 1$ we have
    \[
    2k(d-n) + 1 > 2 \cdot 1 \cdot (2n - 2 - n) + 1 = 2n - 3,
    \]
    and therefore condition \eqref{eq:condition-degeneration} for having zero differentials is always satisfied.
    Similarly, if $d < n/2$ then
    \[
    2k(d-n) + 1 < 2k (-n/2) + 1 = -kn + 1.
    \]
    Hence for $k \geq 2$ this is strictly smaller than $-2n+3$ and condition \eqref{eq:condition-degeneration} is satisfied. If $k = 1$ then $2k(d-n) + 1 < -n+1$, so the only option for \eqref{eq:condition-degeneration} to not hold is $2(d-n) + 1 = -2n + 3$, which implies $d = 1$.
\end{proof}

\section{Cohomology of contact loci} 
\label{sec:contact}

\subsection{Preliminaries}

Let $m$ be a positive integer. 
Recall that an \emph{$m$-jet} in $\CC^n$ is a morphism 
\[
    \gamma: \allowbreak \Spec \CC[t]/(t^{m+1}) \to \mathbb{A}_\CC^n = \Spec \CC[x_1, \ldots, x_n].
\]
The morphism $\gamma$ is uniquely determined by the images of the variables $x_i$ under the associated ring map $\gamma^\sharp: \CC[x_1, \ldots, x_n] \to \CC[t]/(t^{m+1})$. 
We denote those images by
\[
    \gamma^i(t) = \sum_{j=0}^{m} \gamma^i_j t^j \coloneqq \gamma^\sharp(x_i) \in \CC[t]/(t^{m+1}), \quad \text{ where } \gamma^i_j \in \mathbb{C}.
\]
We group the coefficients of each power of $t$ in a vector $\gamma_j \coloneqq (\gamma_j^1, \ldots, \gamma_j^n) \in \CC^n$, and identify the $m$-jet $\gamma$ with the corresponding truncated power series with coefficients in $\CC^n$, i.e. $\gamma(t) = \sum_{j=0}^{m} \gamma_j t^j$.
Similarly, an \emph{arc} in $\CC^n$ is a morphism $\gamma: \allowbreak \Spec \CC \llbracket t \rrbracket \to \mathbb{A}_\CC^n = \Spec \CC[x_1, \ldots, x_n].$

Recall that there is a $\CC$-scheme $\LL_m \coloneqq \LL_m(\CC^n) \coloneqq \Spec \CC[\{\gamma^i_j\}^{i=1,\ldots,n}_{j=0,\ldots,m}]$, called the \emph{$m$-jet space} of $\CC^n$, whose closed points are precisely the $m$-jets in $\CC^n$, see \cite{ein2004}, \cite{jet-arc}. 
For $m' < m$ there is a natural truncation map $\pi^m_{m'}: \LL_m \to \LL_{m'}$.
The inverse limit exists, and we denote it by $\LL_\infty \coloneqq \LL_\infty(\CC^n) \coloneqq \varprojlim_m \LL_m$ and call it the \emph{arc space} of $\mathbb{C}^n$.
By \cite[Thm. 1.1]{bhatt2016}, the closed points of $\mathcal{L}_\infty$ are in bijection with the arcs in $\CC^n$.

For any polynomial $f \in \mathbb{C}[x_1, \ldots, x_n]$, the \emph{restricted $m$-contact locus} associated to $f$ is 
\[
\X_m \coloneqq \X_m(f, 0) \coloneqq \{ \gamma \in \LL_m(\CC^n) \mid \gamma(0)=0, f(\gamma(t)) \equiv t^m\ (\text{mod } t^{m+1})\}.
\]

Note that the contact loci are equipped with the analytic topology coming from the inclusion $\X_m \subset \LL_m \cong \CC^{n(m+1)}$. 
Whenever we refer to the cohomology of $\X_m$ (e.g. in the statement of the arc-Floer conjecture) we are referring to the space equipped with the analytic topology. 
Unless otherwise stated, all cohomology groups should be understood to have integral coefficients. 

\subsection{Filtering by the order}

We are interested in the case where $f \in \CC \llbracket x_1,\ldots,x_n \rrbracket$ is a semihomogeneous power series of degree $d$, i.e. $f = h + F$ with $h$ homogeneous of degree $d$ and having $0$ as an isolated critical point and $F$ of order strictly greater than $d$, see Definition~\ref{def:semihomogeneous}. 
The initial term $h$ defines a smooth hypersurface in $\PP^{n-1}$ which we denote by $S$, just as we did in Section \ref{sec:floer}.
We consider the \emph{increasing} filtration $F_\bullet \LL_m$ on the jet space by the \emph{negative} order of the jets, i.e.
\begin{equation} \label{eq:order-filtration}
    F_p \LL_m \coloneqq \{\gamma \in \LL_m \mid \ord_t(\gamma) \geq -p\},
\end{equation}
where $\ord_t(\gamma) \coloneqq \min\{i \in \ZZ_{\geq 0} \mid \gamma_i \neq 0\}$. The ``graded pieces'' of this filtration are
\[
    F_{(p)}\LL_m \coloneqq F_p \LL_m \setminus F_{p-1} \LL_m = \{\gamma \in \LL_m \mid \ord_t(\gamma) = -p\}.
\]
Since the filtration $F_\bullet \LL_m$ is concentrated in negative degrees, we will denote $\rho \coloneqq -p$ to simplify the notation, so that the jets in $F_{(p)}$ are those of order $\rho$.
Consider the filtration $F_\bullet \X_m$ on the contact locus induced by $F_\bullet \LL_m$, whose graded pieces are $F_{(p)} \X_m = \X_m \cap F_{(p)} \LL_m$.

\begin{proposition} \label{prop:filtered-topology}
    The graded pieces $F_{(p)} \X_m$ have the following description:
    \begin{enumerate}[label=(\roman{*})]
        \item \label{topology-empty} If $p \notin [-m/d, -1] \cap \ZZ$, then $F_{(p)} \X_m = \varnothing$. Otherwise, (ii)-(iv) hold.

        \item \label{zero-free-variables} For $\gamma \in F_{(p)} \X_m,\ p \in [-m/d, -1] \cap \ZZ$, the variables $\gamma_0, \ldots, \gamma_{\rho-1}$ are zero and the variables $\gamma_{m-(d-1)\rho+1},\ldots,\gamma_m$ are not subject to any equation. Formally,
        \[
        \pi^m_{\rho-1}(F_{(p)} \X_m) = \{0\}, \quad
        F_{(p)} \X_m = (\pi^m_{m-(d-1)\rho})^{-1}(\pi^m_{m-(d-1)\rho}(F_{(p)} \X_m)).
        \]

        \item \label{lowest-order} The variable $\gamma_\rho$ is only constrained by one equation; namely $f(\gamma_\rho) = 0$ if $p \in (-m/d, -1]$, and $f(\gamma_\rho) = 1$ if $p = -m/d$. Formally,
        \[
        \pi^m_\rho(F_{(p)} \X_m) \cong \begin{cases}
            CS^\circ & \text{ if } -m/d < p \leq -1, \\
            M_h & \text{ if $d$ divides $m$ and } p = -m/d,
        \end{cases}
        \]
        where $CS^\circ = \{x \in \mathbb{C}^n \mid f(x) = 0\} \setminus \{0\}$ is the affine cone over the projective hypersurface $S$ with the origin removed, and $M_h = \{x \in \mathbb{C}^n \mid h(x) = 1\}$ is the affine Milnor fiber of $h$.
        
        \item \label{truncation-bundle} For $a \in [\rho + 1, m-(d-1)\rho] \cap \ZZ$, the variable $\gamma_a$ is restricted to a hyperplane that depends holomorphically in the lower-order variables and in fact this hyperplane is, after translating it to the origin, the tangent hyperplane of $\pi^m_\rho(F_{(p)} \X_m)$ at $\gamma_\rho$. 
        Formally, the restriction of the truncation map
        \[
        \pi^a_{a-1}|: \pi^m_a(F_{(p)} \X_m) \to \pi^m_{a-1}(F_{(p)} \X_m)
        \]
        is a complex affine bundle, whose model vector bundle is the pullback by $\pi^{a-1}_\rho$ of the tangent bundle of $\pi^m_\rho(F_{(p)} \X_m)$.
    \end{enumerate}
\end{proposition}

\begin{proof}
    Note that by definition there are no jets $\gamma \in \LL_m$ of negative degree. 
    If $\ord_t \gamma = 0$ then $\gamma(0) \neq 0$ and hence $\gamma \notin \X_m$. 
    And if $\ord_t \gamma = \alpha > m/d$ then $\gamma(t) = t^\alpha \delta(t)$ for a jet $\delta$ with $\delta(0) \neq 0$, thus
    \[
        \ord_t f(\gamma(t)) = \ord_t (t^{d\alpha} h(\delta(t)) + F(t^\alpha \delta(t))) \geq \min\{ d\alpha, \ord_t F(t^\alpha \delta(t)) \} = d\alpha > m
    \]
    by homogeneity of $h$, concluding that $\gamma \notin \X_m$. This proves (i).

    Recall that the jets $\gamma \in F_{(p)} \LL_m$ are of the form $\gamma(t) = \gamma_\rho t^\rho + \cdots + \gamma_m t^m$ with $\gamma_i \in \CC^n$ for $i \in \{\rho, \ldots, m\}$ and $\gamma_\rho \neq 0$. 
    Consider the Taylor expansion for $f(\gamma(t))$ around $\gamma_\rho t^\rho$: 
    \[
    f(\gamma(t)) = \sum_{k = 0}^{d} \dfrac{1}{k!} D^k f (\gamma_\rho t^\rho) \cdot (\gamma_{\rho+1} t^{\rho+1} + \cdots + \gamma_m t^m, \ldots, \gamma_{\rho+1} t^{\rho+1} + \cdots + \gamma_m t^m).
    \]
    Using linearity of the derivative, the fact that $D^k h(x) \cdot (v_1, \ldots, v_k)$ is homogeneous of degree $d-k$ in the variable $x$ and linear in each variable $v_i$, and the fact that $\ord_t D^k F(\gamma_\rho t^\rho) > (d-k)\rho$, we can group the coefficients of each power of $t$ in the expression above. 
    Doing this, we obtain
    \[
        \mathrm{Coef}(f(\gamma(t)), t^{d\rho}) = h(\gamma_\rho)
    \]
    and for every $d\rho < \alpha \leq m$ we get
    \begin{align} \label{coef-equation}
    \mathrm{Coef}(f(\gamma(t)), t^\alpha) 
    &= Dh(\gamma_\rho) \cdot \gamma_{\alpha - (d-1)\rho} + \text{(terms involving only $\gamma_\rho, \ldots, \gamma_{\alpha - (d-1)\rho - 1}$).}
    \end{align}

    The defining equation of the contact locus, which is $f(\gamma(t)) \equiv t^m (\mathrm{mod}\ t^{m+1})$, may therefore be split into $m - d\rho + 1$ non-trivial equations in the variables $\gamma_k$, namely
    \[
        \mathrm{Coef}(f(\gamma(t)), t^\alpha) = 0 \text{ for } d\rho \leq \alpha < m
        \quad \text{ and } \quad
        \mathrm{Coef}(f(\gamma(t)), t^m) = 1.
    \]
    Observe that the variables $\gamma_0, \ldots, \gamma_{\rho-1}$ are zero by definition of $F_{(p)} \X_m$, and that the variables $\gamma_{m-(d-1)\rho+1}, \ldots, \gamma_m \in \CC^n$ do not appear in the equations. This proves (ii).

    If $m = d\rho$, then the only non-trivial equation is $\mathrm{Coef}(f(\gamma(t)), t^{m}) = 1$, i.e. $h(\gamma_\rho) = 1$. 
    This equation defines the global Milnor fiber $M_h$, see \cite[Exercise 3.1.13]{dimca1992}. 
    On the other hand, if $m > d\rho$ then the first non-trivial equation is $\mathrm{Coef}(f(\gamma(t)), t^{d\rho}) = 0$, i.e. $h(\gamma_\rho) = 0$, which defines the affine cone over $S$.
    Since $\gamma \in F_{(p)} \X_m$ we must also have $\gamma_\rho \neq 0$, so putting both equations together we obtain $\gamma_\rho \in CS^\circ$. As we will see next, all the other equations have solutions for every $\rho \in CS^\circ$, and then we will have proved (iii).

    Finally, we see from \eqref{coef-equation} that for every $a \in [\rho + 1, m - (d-1)\rho]$ the first equation in which $\gamma_a$ appears is $\mathrm{Coef}(f(\gamma(t)), t^{a + (d-1)\rho}) = 0$ or $1$. 
    This is an affine equation in $\gamma_a$ whose coefficients depend holomorphically on $\gamma_\rho, \ldots, \gamma_{a-1}$. 
    Furthermore, the linear form $Dh(\gamma_\rho)$ is nonzero because $\gamma_\rho \neq 0$ is a smooth point of $h$, and the equation $Dh(\gamma_\rho) = 0$ defines the tangent bundle of $\pi^m_\rho(F_{(p)} \X_m)$. This proves (iv).
\end{proof}

\subsection{Comparing filtrations} \label{subsec:comparing-filtrations}

We are going to prove that the filtration $F_\bullet \X_m$ is equivalent to the filtration described in \cite[Lemma 2.3]{coho-contact}, for the resolution $\mu: X \to \CC^n$ introduced in \S\ref{subsec:m-sep-resolution} and an ample divisor $H$ with the properties of Lemma~\ref{b_E-monotone}. 
Let us start by reviewing the definition of this other filtration.

Recall that we are denoting by $\mathcal{E}$ the set of irreducible components of the exceptional divisor $(f \circ \mu)^{-1}(0)$. 
Recall also that given an \emph{arc} $\sigma \in \LL_\infty(\CC^n)$ whose generic point is not contained in $0 \in \mathbb{C}^n$, there exists a unique lift $\widetilde{\sigma} \in \LL_\infty(X)$ (i.e. $\sigma = \mu \circ \widetilde{\sigma}$) by the valuative criteria for properness and separatedness.
Finally, recall that, since $\CC^n$ is smooth, the truncation morphism $\pi_m: \LL_\infty(\CC^n) \to \LL_m(\CC^n)$ is locally trivial.
For any $E \in \mathcal{E}$ let $\X_{m,E}$ be the set of $\gamma \in \X_m$ such that there exists an arc $\sigma \in \LL_\infty(\CC^n)$ with $\pi_m(\sigma) = \gamma$ and such that $\widetilde{\sigma}(0) \in E$. Since the log resolution $\mu$ is $m$-separating, we have a decomposition
\[
    \X_m = \bigsqcup_{E \in \mathcal{E}} \X_{m,E},
\]
compare with \cite[Lemma 2.1]{coho-contact}. Fix an ample divisor $H = \sum_{E \in \mathcal{E}} b_E E$ on $X$ with the properties of Lemma~\ref{b_E-monotone}, and define an \emph{increasing} filtration $G_\bullet \X_m$ by
\[
    G_p \X_m \coloneqq \bigsqcup_{\substack{E \in \mathcal{E}, \ N_E \mid m \\ mb_E/N_E \leq p}} \X_{m,E}.
\]

\begin{proposition} \label{prop:filtrations-equiv}
    The filtrations $F_\bullet \X_m$ and $G_\bullet \X_m$ are equivalent in the following sense: 
    there exists an increasing function $\iota: \ZZ \to \ZZ$ such that $F_p \X_m = G_{\iota(p)} \X_m$ for all $p \in \ZZ$.
\end{proposition}

\begin{proof}
    Recall from the proof of Theorem~\ref{thm:degeneration-mclss} that we may label the $m$-divisors in the resolution $\mu$ (i.e. the $E \in \mathcal{E}$ such that $N_E$ divides $m$) by $E_{-\lfloor m/d \rfloor}, \ldots, E_{-1}, E_0$ in such a way that $i < j$ implies $b_i/N_i < b_j/N_j$. 
    Therefore the function $\iota(i) = mb_i/N_i$ for $i \in [- m/d, 0] \cap \ZZ$ is increasing, and we may extend it arbitrarily to an increasing function $\ZZ \to \ZZ$.
    All we need to show is that the graded pieces of $F_\bullet$ and $G_\bullet$ are equal, i.e. $F_{(p)} \X_m = G_{(\iota(p))} \X_m$ for all $p \in \ZZ$.
    This amounts to showing that an jet $\gamma \in \X_m$ has order $\rho$ if and only if there exists an arc $\sigma \in \LL_\infty(\CC^n)$ such that $\pi_m(\sigma) = \gamma$ and $\widetilde{\sigma}(0) \in E_{-\rho}$.

    Recall that $\ord_t \gamma$ is the interection multiplicity of $\gamma$ with a general hyperplane $L \subset \CC^n$ through the origin. Note that a hyperplane $L$ meets the origin with multiplicity $1$. An easy induction then shows that
    \[
        \mu^* L = \widetilde{L} + \sum_{\substack{E_{(\kappa, r)} \in \mathcal{E} \\ (\kappa, r) \neq (1,0)}} r E_{(\kappa, r)}, 
    \]
    where $\widetilde{L}$ is the strict transform of $L$, see \S\ref{subsec:m-sep-resolution} for the notation $E_{(\kappa, r)}$. 
    Let $\sigma$ be any arc such that $\pi_m(\sigma) = \gamma$ and let $E$ be the unique divisor such that $\widetilde{\sigma}(0) \in E$, whose associated coprime pair we denote by $(\kappa_E, r_E)$. Then
    \[
    \ord_t \gamma =
    \gamma \cdot L = 
    \sigma \cdot L = 
    \widetilde{\sigma} \cdot (\mu^* L) = 
    r_E (\widetilde{\sigma} \cdot E)
    \]
    because for a general $L$ we have $\widetilde{\sigma} \cdot \widetilde{L} = 0$. On the other hand,
    \[
    m = \gamma \cdot f^{-1}(0) =
    \sigma \cdot f^{-1}(0) =
    \widetilde{\sigma} \cdot (\mu^*(f^{-1}(0))) =
    N_E (\widetilde{\sigma} \cdot E).
    \]
    In particular, $E$ is an $m$-divisor, so there is an $i \in \{ - \lfloor m/d \rfloor, \ldots, -1, 0\}$ such that $E = E_i$.
    Recall from \eqref{eq:kappa-r-formula} that the coprime pair corresponding to the divisor $E_i$ is $(m+id, -i)/\mathrm{gcd}$. 
    Putting the previous two equations together and using \eqref{eq:N-nu-formulas} we conclude
    \[
        \ord_t \gamma = \frac{r_E m}{N_E} = \frac{r_E m}{\kappa_E + r_E d} = \frac{-im}{(m + id) - id} = -i.
    \]
    In other words, $\ord_t \gamma = \rho$ if and only if $\sigma$ lifts to $E_{-\rho}$, as we wanted to show.
\end{proof}

\subsection{Degeneration of the order spectral sequence}
\label{subsec:degeneration-clss}

Since $F_\bullet \X_m$ is an increasing filtration by closed subsets, there is a spectral sequence in cohomology with compact support
\begin{equation} \label{eq:clss}
\phantom{}_{F_\bullet \X_m}E_1^{p,q} = H^{p+q}_c(F_{(p)} \X_m) \implies H^{p+q}_c(\X_m),
\end{equation}
see \cite[\href{https://stacks.math.columbia.edu/tag/012K}{Tag 012K}]{stacks-project} or \cite[\S 3-\S 4]{arapura2005}.
In fact, by Proposition~\ref{prop:filtrations-equiv} this spectral sequence coincides, up to a relabeling of the entries, with the spectral sequence $\{\phantom{}_{\text{BFLN}}E_\ell^{p,q}\}_{\ell\geq 0}$ given by Budur et al. \cite[Thm. 1.1]{coho-contact}. 
In turn, the first page of the latter spectral sequence is isomorphic, up to a shift, to the first page of the McLean spectral sequence \eqref{eq:mclss}, see \cite[Rmk. 1.4]{coho-contact}. 
More precisely, using the increasing function $\iota$ from Proposition \ref{prop:filtrations-equiv} and taking into account the index considerations of Remark~\ref{rmk:index-fix}, we have
\begin{equation} \label{eq:spectral-sequence-comparison}
    \phantom{}_\text{McLean} E_1^{\iota(p),q} \cong \phantom{}_\text{BFLN} E_1^{\iota(p),q + (n-1)(2m+1)} \cong \phantom{}_{F^\bullet \X_m} E_1^{p, q + (n-1)(2m+1) + \iota(p) - p}.
\end{equation}
Taking this into account, the \emph{arc-Floer conjecture} was stated in \cite[Conj. 1.5]{coho-contact}. With the corrected shift, the conjecture is
\begin{equation} \tag{AFC} \label{eq:arc-floer-conjecture}
    \HF^\bullet(\phi^m, +) \cong H_c^{\bullet + (n-1)(2m+1)}(\X_m) \quad \text{for every integer } m \geq 1.
\end{equation}

Since the proof of degeneration in Theorem~\ref{thm:degeneration-mclss} only uses the fact that specific terms in the first page are nonzero, we could use exactly the same proof to prove the degeneration of \eqref{eq:clss} in some cases.
Nevertheless, we can exploit the topology of the contact loci that we computed in Proposition~\ref{prop:filtered-topology} to prove degeneration in all cases.
Another setting in which $\phantom{}_\text{BFLN} E_1^{p,q}$ is known to degenerate at the first page is the case of hyperplane arrangements \cite[Prop. 7.4]{contact-hyperplane}; note that this is the case of a homogeneous singularity of degree $d=1$ without the assumption of having an isolated critical point at $0$.

Let us start by introducing an auxiliary space that will appear in the proof of degeneration. 
From now on we will assume that $d \geq 2$ (if $d=1$, then the arguments in this section do not apply, and in fact it is not hard to see that the spectral sequence \eqref{eq:clss} does not degenerate).
Let $p, q, \ell$ be integers, such that $\ell \geq 1$ and such that $F_{(p)} \X_m \neq \varnothing$ and $F_{(p + \ell)} \X_m \neq \varnothing$. 
These integers will eventually be the indices of the differential $d_\ell^{p,q}$. 
Let $k \coloneqq m - (d-1)(\rho - \ell)$, which is the integer such that $\gamma_k$ is the last variable subject to an equation inside $F_{(p+\ell)} \X_m$, according to Proposition~\ref{prop:filtered-topology}.\ref{zero-free-variables}.
Define
\begin{equation} \label{eq:X-definition}
    Z \coloneqq (\pi^m_{k-1})^{-1}(\pi^m_{k-1}(\X_m)),
\end{equation}
i.e. $Z$ is the same space as $\X_m$ up to the variable $\gamma_{k-1}$, but the variables $\gamma_k, \ldots, \gamma_m$ are set to be free.
Note that we have a closed inclusion $i:\X_m \hookrightarrow Z \subset \LL_m$ that respects the filtrations $F_\bullet \X_m$ and $F_\bullet Z = Z \cap F_\bullet \LL_m$. Let us study how the latter inclusion looks like on each graded piece.

\begin{lemma} \label{lem:Z-filtration-low-degrees}
    $F_{(r)} Z = F_{(r)} \X_m$ for all $r < p + \ell$.
\end{lemma}

\begin{proof}
    We have
    \[
        \pi^m_{k-1}(F_{(r)} Z) = 
        F_{(r)} (\pi^m_{k-1}(Z)) = 
        F_{(r)} (\pi^m_{k-1}(\X_m)) = 
        \pi^m_{k-1}(F_{(r)} \X_m)
    \]
    where the first and last equalities hold because the truncation maps respect the order, and the middle equality holds 
    by definition of $Z$. On the other hand,
    \[
        (\pi^m_{k-1})^{-1}(\pi^m_{k-1}(F_{(r)} Z)) = F_{(r)} Z
        \qquad \text{and} \qquad
        (\pi^m_{k-1})^{-1}(\pi^m_{k-1}(F_{(r)} \X_m)) = F_{(r)} \X_m,
    \]
    where the left equality holds by definition of $Z$ and the right equality holds by Proposition~\ref{prop:filtered-topology}.\ref{zero-free-variables}, using the fact that $\pi^m_{m-(d-1)(-r)} = \pi^{k-1}_{m-(d-1)(-r)} \circ \pi^m_{k-1}$. Observe that $\pi^{k-1}_{m-(d-1)(-r)}$ is well defined if and only if $k > m-(d-1)(-r)$, which holds if and only if $r < p+\ell$, and this explains our assumption. This finishes the proof.
\end{proof}

\begin{lemma} \label{lem:Z-p-ell-bundle-inclusion}
    We have the following description of the inclusion $i: F_{(p + \ell)} \X_m \hookrightarrow F_{(p + \ell)} Z$:
    \begin{enumerate}[label=(\roman{*})]
        \item \label{inclusion-comes-from-below} The following diagram commutes and the vertical arrows are isomorphisms:
        \[\begin{tikzcd}
        	{F_{(p+\ell)} \X_m} & {F_{(p+\ell)} Z} \\
        	{\pi^m_k(F_{(p+\ell)} \X_m) \times \CC^{n(m-k)}} & {\pi^m_k(F_{(p+\ell)} Z) \times \CC^{n(m-k)}}
        	\arrow["i", hook, from=1-1, to=1-2]
        	\arrow["{(\pi^m_k, \pr_{\gamma_{k+1}, \ldots, \gamma_m})}"', "\cong", from=1-1, to=2-1]
        	\arrow["{(\pi^m_k, \pr_{\gamma_{k+1}, \ldots, \gamma_m})}", "\cong"', from=1-2, to=2-2]
        	\arrow["{i \times \id}", hook, from=2-1, to=2-2]
        \end{tikzcd}\]
        
        \item \label{below-is-pullback} There is a commutative diagram
        \[\begin{tikzcd}
        	& {T_{CS^\circ}} && {i_{CS^\circ}^*(T_{\mathbb{C}^n})} \\
        	{\pi^m_k(F_{(p+\ell)} \X_m)} && {\pi^m_k(F_{(p+\ell)} Z)} \\
        	& {CS^\circ} && {CS^\circ} \\
        	{\pi^m_{k-1}(F_{(p+\ell)} \X_m)} && {\pi^m_{k-1}(F_{(p+\ell)} Z)}
        	\arrow["(i_{CS^\circ})_*", hook, from=1-2, to=1-4]
        	\arrow[from=1-2, to=3-2]
        	\arrow[from=1-4, to=3-4]
        	\arrow["{(\pi^{k-1}_{\rho-\ell},\ \pr_{\gamma_k})}", from=2-1, to=1-2]
        	\arrow["{\pi^k_{k-1}}"', from=2-1, to=4-1]
        	\arrow["{(\pi^{k-1}_{\rho-\ell},\ \pr_{\gamma_k})}", from=2-3, to=1-4]
        	\arrow[equal, from=3-2, to=3-4]
        	\arrow["{\pi^{k-1}_{\rho-\ell}}", from=4-1, to=3-2]
        	\arrow[equal, from=4-1, to=4-3]
        	\arrow["{\pi^{k-1}_{\rho-\ell}}"', from=4-3, to=3-4]
            \arrow["i" {xshift=2em}, hook, from=2-1, to=2-3, crossing over]
        	\arrow["{\pi^k_{k-1}}" {yshift=1.5em}, from=2-3, to=4-3, crossing over]
        \end{tikzcd}\]
        such that the left and right squares are Cartesian. In other words, the inclusion
        \[
        i: \pi^m_k(F_{(p+\ell)} \X_m) \hookrightarrow \pi^m_k(F_{(p+\ell)} Z)
        \]
        of vector bundles over $\pi^m_{k-1}(F_{(p+\ell)} \X_m)$ is the pullback of the inclusion $T_{CS^\circ} \hookrightarrow i^*_{CS^\circ}(T_{\CC^n})$ of vector bundles over $CS^\circ$ (where $i_{CS^\circ}: CS^\circ \hookrightarrow \mathbb{C}^n$ is the inclusion) under the map
        \[
        \pi^{k-1}_{\rho-\ell}: \pi^m_{k-1}(F_{(p+\ell)} \X_m) \to \pi^m_{\rho-\ell}(F_{(p+\ell)} \X_m) \cong CS^\circ.
        \]
    \end{enumerate}
\end{lemma}

\begin{proof}
    Item (i) is just the statement that
    \[
        (\pi^m_k)^{-1}(\pi^m_k(F_{(r)} Z)) = F_{(r)} Z
        \qquad \text{and} \qquad
        (\pi^m_k)^{-1}(\pi^m_k(F_{(r)} \X_m)) = F_{(r)} \X_m,
    \]
    which follows just like in the proof of Lemma~\ref{lem:Z-filtration-low-degrees}.
    Now consider the diagram from item (ii). The bottom arrow in the front square is just the inclusion, which is an equality by definition of $Z$. Hence it is clear that the front square commutes. The bottom square is just stating 
    \[
    \pi^m_{\rho-\ell}(F_{(p+\ell)} Z) = \pi^m_{\rho-\ell}(F_{(p+\ell)} \X_m) = CS^\circ,
    \]
    which follows by definition of $Z$ and Proposition~\ref{prop:filtered-topology}.\ref{lowest-order}, noting that $p+\ell > p \geq -m/d$ because $F_{(p)} \X_m$ is nonempty.
    The left square is Cartesian by Proposition~\ref{prop:filtered-topology}.\ref{truncation-bundle}. 
    The right square is Cartesian because $i^*_{CS^\circ}(T\CC^n) \to CS^\circ$ is the trivial bundle and
    \[
    \pi^m_k(F_{(p+\ell)} Z) = (\pi^k_{k-1})^{-1}(\pi^m_{k-1}(F_{(p+\ell)} Z))
    \]
    by definition of $Z$. 
    The back square is the usual inclusion of vector bundles given by the differential of $i_{CS^\circ}$.
    The top square clearly commutes.
\end{proof}

We need the following result about the cohomology of vector bundles.

\begin{lemma} \label{lem:pullback-exact-sequence}
    Consider a short exact sequence
    \[\begin{tikzcd}[row sep=scriptsize]
        0 & {V'} & V & {V''} & 0 \\
        && B
        \arrow[from=1-1, to=1-2]
        \arrow["i", from=1-2, to=1-3]
        \arrow["{\xi'}"', from=1-2, to=2-3]
        \arrow["q", from=1-3, to=1-4]
        \arrow["\xi"', from=1-3, to=2-3]
        \arrow[from=1-4, to=1-5]
        \arrow["{\xi''}", from=1-4, to=2-3]
    \end{tikzcd}\]
    of orientable vector bundles over paracompact space $B$, of ranks $n', n$ and $n''$, respectively. Let $e(\zeta), \tau(\zeta)$ denote the Euler class and the Thom class of the vector bundle $\zeta$. Then the following diagram commutes  
    \begin{equation} \label{eq:char-class-diagram}
    \begin{tikzcd}[column sep=6em]
    	{H_c^{k-n}(B)} & {H_c^{k}(V)} \\
    	{H_c^{k-n}(B)} & {H_c^{k}(V')}
    	\arrow["{\xi^*(-) \smile \tau(\xi)}", from=1-1, to=1-2]
    	\arrow["{- \smile e(\xi'')}"', from=1-1, to=2-1]
    	\arrow["{i^*}", from=1-2, to=2-2]
    	\arrow["{(\xi')^*(-) \smile \tau(\xi')}", from=2-1, to=2-2]
    \end{tikzcd}
    \end{equation}
\end{lemma}

\begin{proof}
    Let $p: V \to V'$ be a splitting of the short exact sequence, i.e. $p \circ i = \id_{V'}$, which exists because $B$ admits a Euclidean metric, see \cite[Prob. 2-C]{milnor1974}.
    Note that $\xi = \xi'' \circ p$ is a composition of vector bundles, so by uniqueness of the Thom class \cite[Thm. 10.4]{milnor1974} we have
    \begin{equation} \label{eq:thom-composition}
        \tau(\xi) = \tau(p) \smile p^* \tau(\xi').
    \end{equation}
    Then, for any $\alpha \in H^{k-n}_c(B)$ we have
    \begin{align*}
        i^*(\xi^* \alpha \smile \tau(\xi)) &= i^*(\xi^* \alpha \smile \tau(p) \smile p^* \tau(\xi')) && \text{by \eqref{eq:thom-composition},} \\
        &= (\xi')^* \alpha \smile i^* \tau(p) \smile \tau(\xi') && \text{because $i^*$ distributes over $\smile$,} \\
        &= (\xi')^* \alpha \smile e(p) \smile \tau(\xi') && \text{by definition of the Euler class,} \\
        &= (\xi')^* \alpha \smile (\xi')^* e(\xi'') \smile \tau(\xi') && \text{by functoriality of the Euler class,} \\
        &= (\xi')^* (\alpha \smile e(\xi'')) \smile \tau(\xi') && \text{because $(\xi')^*$ distributes over $\smile$.}
    \end{align*}
\end{proof}

We are finally ready to prove our main result.

\begin{theorem} \label{thm:degeneration-clss}
    If $d \geq 2$, the spectral sequence \eqref{eq:clss} degenerates at the first page. 
\end{theorem}

\begin{proof}
    Fix an integer $\ell \geq 1$ and suppose that we have already proved that the differentials $d_{\ell'}$ are zero for all $1 \leq \ell' < \ell$. We need to show that the differentials
    \[
    \phantom{}_{F_\bullet \X_m} d^{p,q}_\ell: \phantom{}_{F_\bullet \X_m}E^{p,q}_\ell \to \phantom{}_{F_\bullet \X_m}E^{p+\ell, q-\ell+1}_\ell
    \]
    are zero for all $p, q \in \ZZ$. Since all differentials in the previous pages were zero, we have
    \begin{equation} \label{eq:page-ell}
        \phantom{}_{F_\bullet \X_m}E^{p,q}_\ell = \phantom{}_{F_\bullet \X_m}E^{p,q}_1 = H^{p+q}_c(F_{(p)} \X_m) \quad \text{ for all }p, q \in \ZZ.
    \end{equation}
    If either $p \notin [-m/d, -1]$ or $p + \ell \notin [-m/d, -1]$, then $F_{(p)} \X_m = \varnothing$ or $F_{(p+\ell)} \X_m = \varnothing$ by Proposition \ref{prop:filtered-topology}.\ref{topology-empty}, and therefore either the domain or the codomain of $d^{p,q}_\ell$ is zero, so $d^{p,q}_\ell = 0$. 
    Hence we may assume that both $p$ and $p + \ell$ lie in the interval $[-m/d, -1]$.

    Let $k \coloneqq m - (d-1)(\rho - \ell)$ and consider the auxiliary space $Z = (\pi^m_{k-1})^{-1}(\pi^m_{k-1}(\X_m))$ introduced in \eqref{eq:X-definition}, which has an induced filtration $F_\bullet Z$. 
    By functoriality of the spectral sequence associated to a filtration (see \cite[\S 3-\S 4]{arapura2005}, \cite[\href{https://stacks.math.columbia.edu/tag/012O}{Tag 012O}]{stacks-project}), there is a map of spectral sequences $\phantom{}_{F_\bullet Z} E \to \phantom{}_{F_\bullet \X_m} E$ induced by the closed inclusion $i: \X_m \hookrightarrow Z$. 
    As part of the map between $\ell$-th pages, there is a commutative diagram
    \[\begin{tikzcd}
    	{\phantom{aa} H^{p+q}_c(F_{(p)} Z) = \phantom{}_{F_\bullet Z} E_\ell^{p,q}} && {\phantom{}_{F_\bullet Z} E_\ell^{p+\ell,q-\ell+1} \phantom{= H^{p+q+1}_c(F_{(p+\ell)} \X_m)}} \\
    	{H^{p+q}_c(F_{(p)} \X_m) = \phantom{}_{F_\bullet \X_m} E_\ell^{p,q}} && {\phantom{}_{F_\bullet \X_m} E_\ell^{p+\ell,q-\ell+1} = H^{p+q+1}_c(F_{(p+\ell)} \X_m)}
    	\arrow["{\phantom{}_{F_\bullet Z} d_\ell^{p,q}}", from=1-1, to=1-3]
    	\arrow["{i^*}"', from=1-1, to=2-1, shift left=4em]
    	\arrow["{i^*}", from=1-3, to=2-3, shift right=5em]
    	\arrow["{\phantom{}_{F_\bullet \X_m} d_\ell^{p,q}}", from=2-1, to=2-3]
    \end{tikzcd}\]
    where we have used \eqref{eq:page-ell} for the equalities in the bottom row.
    The equality in the top row follows because the differentials involving $\phantom{}_{F_\bullet Z} E_{\ell'}^{p,q}$ for $1 \leq \ell' < \ell$ are precisely the same as those in the spectral sequence $\phantom{}_{F_\bullet \X_m} E$ by Lemma~\ref{lem:Z-filtration-low-degrees}, and hence zero by the induction hypothesis.

    Futhermore, since $F_{(p)} Z = F_{(p)} \X_m$ by Lemma~\ref{lem:Z-filtration-low-degrees}, the left vertical arrow is the identity map. 
    Therefore, to show $\phantom{}_{F_\bullet \X_m} d_\ell^{p,q}$ is zero it is enough to prove that the right vertical arrow is zero.
    Recall that $\phantom{}_{F_\bullet Z} E_\ell^{p+\ell,q-\ell+1}$ is a subquotient of $\phantom{}_{F_\bullet Z} E_1^{p+\ell,q-\ell+1} = H^{p+q+1}_c(F_{(p+\ell)} \X_m)$, and the vertical arrow is simply the map induced by the pullback
    \begin{equation} \label{eq:pullback-Z-Xm}
        i^*: H^{p+q+1}_c(F_{(p+\ell)} Z) \to H^{p+q+1}_c(F_{(p+\ell)} \X_m)
    \end{equation}
    on said subquotient.
    Thus, if we show that the map \eqref{eq:pullback-Z-Xm} is zero then we can conclude that $\phantom{}_{F_\bullet \X_m} d_\ell^{p,q} = 0$, which finishes the induction and the proof.

    By Lemma~\ref{lem:Z-p-ell-bundle-inclusion}.\ref{inclusion-comes-from-below}, it suffices to show that the pullback map in cohomology with compact support induced by the closed inclusion
    \begin{equation} \label{eq:pullback-below}
        i: \pi^m_k(F_{(p+\ell)} \X_m) \hookrightarrow \pi^m_k(F_{(p+\ell)} Z)
    \end{equation}
    is zero. 
    Lemma~\ref{lem:Z-p-ell-bundle-inclusion}.\ref{below-is-pullback} says that \eqref{eq:pullback-below} is an inclusion of vector bundles, hence we may consider the quotient bundle $q: Q \to \pi_{k-1}(F_{(p+\ell)})$, which fits in a short exact sequence
    \[
    0 \to \pi^m_k(F_{(p+\ell)} \X_m) \xrightarrow{i} \pi^m_k(F_{(p+\ell)} Z) \to Q \to 0.
    \]
    By Lemma~\ref{lem:pullback-exact-sequence} we just need to show that the Euler class $e(q)$ is zero (because the horizontal arrows in \eqref{eq:char-class-diagram} are isomorphisms).
    On the other hand, Lemma~\ref{lem:Z-p-ell-bundle-inclusion}.\ref{below-is-pullback} also says that \eqref{eq:pullback-below} is the pullback of the inclusion of vector bundles $T_{CS^\circ} \hookrightarrow i_{CS^\circ}^*(T_{\CC^n})$ under $\pi^{k-1}_{\rho-\ell}$, and in this case we have the short exact sequence
    \[
    0 \to T_{CS^\circ} \to i_{CS^\circ}^*(T_{\CC^n}) \to N_{CS^\circ/\CC^n} \to 0,
    \]
    where $\nu: N_{CS^\circ/\CC^n} \to CS^\circ$ is the normal bundle.
    Therefore, by exactness of the pullback, we know that $q \cong (\pi^{k-1}_{\rho-\ell})^*(\nu)$ and hence $e(q) = (\pi^{k-1}_{\rho-\ell})^*(e(\nu))$. We are going to show that $e(\nu) = 0$.

    A direct application of the nine lemma shows that $\nu$ is the pullback of the normal bundle $N_{S/\mathbb{P}^{n-1}} \to S$ under the natural projection map $\pi: CS^\circ \to S$. 
    Recall that $\pi: CS^\circ \to S$ is the complement of the zero section of the tautological line bundle $\mathcal{O}_S(-1)$, so the Thom-Gysin sequence gives $\pi^*(e(\mathcal{O}_S(-1))) = 0$.
    On the other hand, we know $N_{S/\mathbb{P}^{n-1}} \cong \mathcal{O}_S(d)$ \cite[21.2.J]{vakil}.
    Hence,
    \[
    e(\nu) = \pi^*(e(N_{S/\mathbb{P}^{n-1}})) = \pi^*(e(\mathcal{O}_S(d))) = (-d) \pi^*(e(\mathcal{O}_S(-1))) = 0.
    \]
\end{proof}

\begin{remark}
    Note that it follows from Proposition~\ref{prop:filtered-topology} and Theorem~\ref{thm:degeneration-clss} that both the cohomology groups $H_c^\bullet(\X_m)$ and the classes $[\X_m]$ in the Grothendieck ring of varieties depend only on the homogeneous part $h$ of the semihomogeneous series $f = h + F$, and not on the higher order terms $F$.
\end{remark}

\section{The arc-Floer conjecture}
\label{sec:conclusions}

We are still considering a semihomogeneous germ $f: (\CC^n, 0) \to (\CC,0)$ of degree $d$, see Definition~\ref{def:semihomogeneous}.
As we mentioned previously in (\ref{eq:spectral-sequence-comparison}), the McLean spectral sequence \eqref{eq:mclss} and the spectral sequence associated to the order filtration \eqref{eq:clss} have the same first page up to a relabeling of the columns and a shift. 
Furthermore, in Theorem~\ref{thm:degeneration-mclss} we have shown that McLean's spectral sequence \eqref{eq:mclss} degenerates at the first page under some conditions on the degree $d$ and the number of variables $n$; while in Theorem~\ref{thm:degeneration-clss} we have shown that \eqref{eq:clss} degenerates at the first page unconditionally (and hence so does Budur--Fernández de Bobadilla--L\^{e}--Nguyen's spectral sequence).

While this allows us to conclude that the $E_\infty$ pages of both spectral sequences are isomorphic (under the conditions on $d$ and $n$ and accounting for the shift as in \eqref{eq:spectral-sequence-comparison}), it is not enough to conclude that the groups to which they converge are isomorphic. 
This is a general fact about spectral sequences: a filtered complex $F^\bullet C^\bullet$ induces a filtration in cohomology $F^\bullet H^\bullet$, and (under mild finiteness assumptions that hold in our setting, see \cite[\href{https://stacks.math.columbia.edu/tag/012W}{Tag 012W}]{stacks-project}) the graded pieces are given by the $E_\infty$ page of the associated spectral sequence
\[
F^{(p)} H^{p+q} \cong E_\infty^{p,q}.
\]
Hence, 
reconstructing $H^\bullet$ from the $E_\infty$ page requires understanding some extensions.
Since we are working with integer coefficients, and by Corollary~\ref{cor:homology-computation} there are non-free terms in the $E_1$ page, these extensions might be nontrivial.
Therefore, knowing that the $E_\infty$ pages coincide only gives us a partial result.

\begin{proposition}
    Let $n \geq 3, d \geq 2, m \geq 1$ be integers and suppose that condition \eqref{eq:condition-degeneration} is satisfied.
    Let $f: (\CC^n, 0) \to (\CC,0)$ be an homogeneous polynomial of degree $d$ with an isolated singularity at the origin.
    Then we have isomorphisms between the associated graded groups
    \[
        \mathrm{gr}_F H_c^{\bullet + (n-1)(2m+1)}(\X_m) \cong
        \mathrm{gr}_A \HF^\bullet(\phi^m, +)
    \]
    where $F$ is induced by the order filtration \eqref{eq:order-filtration} on the $m$-contact locus $\X_m$, and $A$ is induced by the action filtration on the Floer complex of the $m$-th iterate of the monodromy $\phi^m$ \cite[(HF3)]{mclean}.
\end{proposition}

\begin{proof}
    This follows from Theorem~\ref{thm:degeneration-mclss}, Theorem~\ref{thm:degeneration-clss} and the discussion above.
\end{proof}

Nevertheless, we have computed the $E_1$ page (and hence the $E_\infty$ page) explicitly.
Therefore, we can give an explicit condition that guarantees that the filtrations in compactly supported cohomology and Floer cohomology have at most one graded piece, so that no extension problems appear.

\begin{theorem} \label{thm:arc-floer-semihomogeneous}
    Let $n \geq 3, d \geq 2, m \geq 1$ be integers and suppose that condition \eqref{eq:condition-degeneration} is satisfied. Furthermore, suppose that
    \begin{equation} \label{eq:condition-filtration}
        2k(d-n) \notin \{0, \pm (n-2), \pm (n-1)\} \text{ for all } k \in [1, m/d) \cap \ZZ.
    \end{equation}
    Let $f: (\CC^n, 0) \to (\CC,0)$ be an homogeneous polynomial of degree $d$ with an isolated singularity at the origin.
    Then we have isomorphisms
    \[
        H_c^{\bullet + (n-1)(2m+1)}(\X_m) \cong
        \HF^\bullet(\phi^m, +).
    \]
    In particular, if $2 \leq d < n/2$ or $d > 2n - 2$ then conditions \eqref{eq:condition-degeneration} and \eqref{eq:condition-filtration} hold for all $m \geq 1$, and therefore the arc-Floer conjecture holds.
\end{theorem}

\begin{proof}
    We just need to show that condition \eqref{eq:condition-filtration} implies that for every total degree $r$ there is at most one index $p$ such that $E_1^{p, r-p} \neq 0$.
    As in the proof of Theorem~\ref{thm:degeneration-mclss}, let $E_i,\ i \in \{-\lfloor m/d \rfloor, \ldots, -1\}$ be the exceptional $m$-divisors of the log resolution $\mu: X \to \CC^n$ and let $p_i = mb_i/N_i$.

    Suppose that we have $i < j$ such that $E_1^{p_i, q_i} \neq 0$, $E_1^{p_j, q_j} \neq 0$ and $p_i + q_i = p_j + q_j = r$. By Corollary~\ref{cor:homology-computation}, this implies
    \[
    \begin{cases}
        n-1-(p_i+q_i)-2i(d-n) \in \{ 0, n-2, n-1, 2n-3 \}, \\
        n-1-(p_j+q_j)-2j(d-n) \in \{ 0, n-2, n-1, 2n-3 \}.
    \end{cases}   
    \]
    and taking the difference we see
    \[
        2(j-i)(d-n) \in \{ 0, \pm 1, \pm(n-1), \pm(n-2), \pm(2n-3) \}.
    \]
    Since $2(j-i)(d-n)$ is even, we may ignore the cases $\pm 1$ and $\pm(2n-3)$. The result follows. That the bounds $2 \leq d < n/2$ and $d > 2n-2$ are enough to satisfy condition \eqref{eq:condition-filtration} is seen in exactly the same way as in the proof of Theorem~\ref{thm:degeneration-mclss}.
\end{proof}

\begin{figure}[ht]
    \centering
    \makebox[\textwidth][c]{\resizebox{0.9\linewidth}{!}{
        \includegraphics{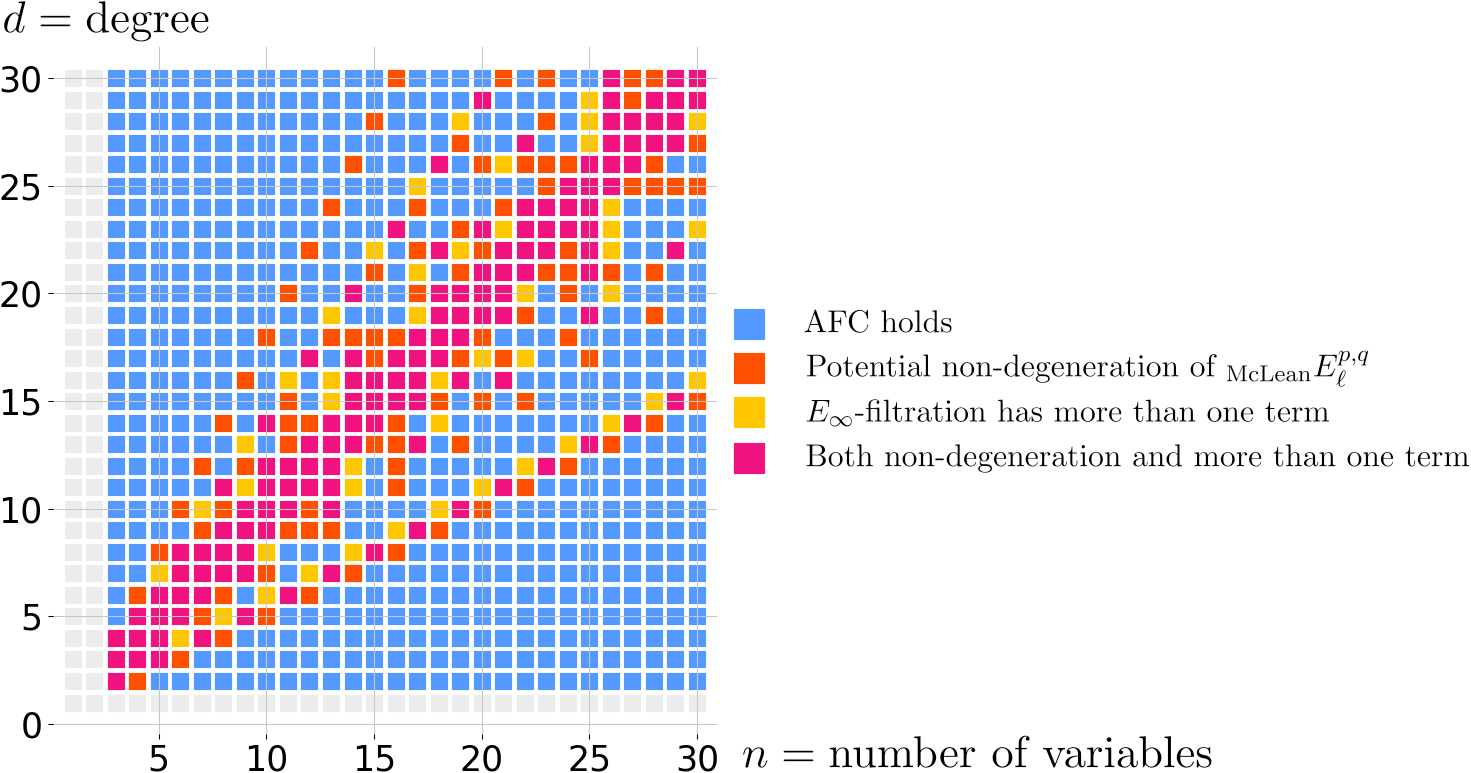}
    }}
    \caption{Plot showing the pairs $(n,d)$ such that there is some $m$ for which condition \eqref{eq:condition-filtration} fails (in orange), condition \eqref{eq:condition-degeneration} fails (in yellow) or both of them fail (in pink). The points in blue are the pairs for which we have shown that the arc-Floer conjecture holds. }
    \label{fig:scatter-plot}
\end{figure}

\section{The embedded Nash problem}
\label{sec:embedded-nash}

\subsection{Preliminaries}

Recall once again that an arc in $\CC^n$ is a morphism $\gamma: \allowbreak \Spec \CC \llbracket t \rrbracket \to \mathbb{A}_\CC^n = \Spec \CC[x_1, \ldots, x_n].$
There is a $\CC$-scheme $\mathcal{L}_\infty \coloneqq \mathcal{L}_\infty(\CC^n)$, called the \emph{arc space} of $\CC^n$, whose closed points are arcs in $\CC^n$. 
It is equipped with canonical \emph{truncation morphisms} $\pi^\infty_m: \mathcal{L}_\infty \to \mathcal{L}_m$ for every $m \in \mathbb{Z}_{\geq 0}$. 
For a polynomial $f \in \mathbb{C}[x_1, \ldots, x_n]$ we define the \emph{unrestricted $m$-contact locus} as
\[
    \X_m^\infty = \{ \gamma \in \mathcal{L}_\infty \mid \gamma(0) = 0,\ \mathrm{ord}_t (f(\gamma(t))) = m \}.
\]

We now give the precise statement of the embedded Nash problem. 
Let $\mu: X \to \CC^n$ be an $m$-separating log resolution of the triple $(\CC^n, f, 0)$ (by definition this is an $m$-separating log resolution of the pair $(\CC^n, f)$ with the additional requirement that $\mu^{-1}(0)$ be a normal crossing divisor), and denote the total transform by $(f \circ \mu)^{-1}(0) = \sum_{E \in \mathcal{E}} N_E E$ as we did in Section \ref{sec:floer}.
For any arc $\gamma \in \X_m^\infty$, the image of its generic point is not contained in $f^{-1}(0)$, and therefore we may lift $\gamma$ to a unique arc in the resolution $\widetilde{\gamma} \in \mathcal{L}_\infty(X)$, meaning $\gamma = \mu \circ \widetilde{\gamma}$.
Note that
\[
    m = \mathrm{ord}_t (f(\gamma(t))) = \mathrm{ord}_t (f(\mu(\widetilde{\gamma}(t)))) = \sum_{E \in \mathcal{E}} N_E \cdot \mathrm{ord}_E(\widetilde{\gamma}).
\]
Since the values $\mathrm{ord}_E(\widetilde{\gamma})$ are a non-negative integers, the $m$-separating condition 
\[
    N_E + N_F \leq m \implies E \cap F = \varnothing
\]
implies that there is a unique $E \in \mathcal{E}$ such that $\mathrm{ord}_E(\widetilde{\gamma}) > 0$, and furthermore $N_E$ divides $m$, i.e. $E$ is an $m$-divisor.
Following \cite{coho-contact}, we define
\[
    \X_{m,E}^\infty \coloneqq \{ \gamma \in \X_m^\infty \mid \mathrm{ord}_E(\widetilde{\gamma}) > 0 \} = \{ \gamma \in \X_m^\infty \mid \widetilde{\gamma}(0) \in E \},
\]
and obtain the following decomposition (compare with \S\ref{subsec:comparing-filtrations})
\begin{equation} \label{eq:XmE-decomposition}
    \X_m^\infty = \bigsqcup_{\substack{E \in \mathcal{E} \\ N_E \mid m}} \X_{m,E}^\infty.
\end{equation}

\begin{proposition}[{\cite[Thm. 2.1]{ein2004}, \cite[Lemma 2.6]{coho-contact}}] \label{prop:codimensions}
    For every \emph{exceptional} $E \in \mathcal{E}$ such that $N_E \mid m$, $\X_{m,E}^\infty$ is irreducible and has codimension $m \nu_E/N_E$ in $\mathcal{L}_\infty(\CC^n)$.
\end{proposition}

Hence, for every irreducible component $C$ of $\X_m^\infty$, there is a unique exceptional $E \in \mathcal{E}$ such that $C = \overline{\X_{m,E}^\infty}$. 
Of course, the converse is not true: there are divisors $F \in \mathcal{E}$ for which there exists another exceptional divisor $F \neq E \in \mathcal{E}$ such that $\X_{m,F}^\infty \subset \overline{\X_{m,E}^\infty}$. The embedded nash problem is the following:
\begin{equation} \tag{ENP} \label{eq:embedded-nash-problem}
    \text{describe the $E \in \mathcal{E}$ such that $\overline{\X_{m,E}^\infty}$ is an irreducible component of $\X_m^\infty$}
\end{equation}

In analogy to de Fernex and Docampo's solution to the classical Nash problem \cite{defernex2016}, the authors of \cite{irred-curve} gave a result that gives ``bounds'' for the components of the contact locus. 
Before stating it, we need to introduce some definitions.

\begin{definition} \label{def:m-valuations}
    Let $Y$ be a smooth complex algebraic variety, $D$ a nonzero effective divisor on $Y$, and $\Sigma \subset D_\mathrm{red}$ a nonempty closed subset.
    \begin{enumerate}[label=(\roman*)]
        \item An \emph{$m$-valuation} of $(Y, D, \Sigma)$ is a divisorial valuation $v$ on $Y$ given by an $m$-divisor on some $m$-separating log resolution of the triple -- that is, $v = \mathrm{ord}_E$ for some prime divisor $E$ on an $m$-separating log resolution $\mu: X \to (Y, D, \Sigma)$ with $N_E \mid m$.
        
        \item Assume $D$ is reduced. A \emph{dlt $m$-valuation} of $(Y, D, \Sigma)$ is an $m$-valuation given by a prime divisor lying over $\Sigma$ in a dlt modification of $(Y, D)$, see \cite[Def. 2.4]{irred-curve}.

        \item A \emph{contact $m$-valuation} of $(Y, D, \Sigma)$ is an $m$-valuation $v = \mathrm{ord}_E$ such that $\overline{\X_{m,E}^\infty}$ is an irreducible component of the contact locus $\X_m^\infty$.
    
        \item An \emph{essential $m$-valuation} of $(Y, D, \Sigma)$ is an $m$-valuation $v$ whose center on every $m$-separating log resolution is a prime divisor lying over $\Sigma$.
    \end{enumerate}
\end{definition}

\begin{proposition}[{\cite[Thm. 1.13]{irred-curve}}] \label{prop:embedded-nash}
    Let $(Y, D, \Sigma)$ be as before, and $m \in \mathbb{Z}_{> 0}$.
    Then:
    \[
        \{ \text{dlt $m$-valuations} \} \subset \{ \text{contact $m$-valuations} \} \subset \{ \text{essential $m$-valuations} \}.
    \]
\end{proposition}

\subsection{Dlt, contact and essential valuations of semihomogeneous singularities}

Our goal is to describe the three sets of Proposition \ref{prop:embedded-nash} in the case where $Y = \CC^n$, $D = \mathrm{div}(f)$ is a semihomogeneous singularity and $\Sigma = \{ 0 \}$. 
To do so, we will use the $m$-separating resolution $\mu: X \to \CC^n$ that we introduced in \S\ref{subsec:m-sep-resolution}.
The key facts about $\mu$ we need to recall are: that we have constructed it as an explicit sequence of blow-ups \eqref{eq:sequence-blow-ups}; that the dual graph of $\mu$ is a chain; and that if we label its $m$-divisors as $E_{-\lfloor m/d \rfloor}, \ldots, E_{-1}, E_0$, where $E_0$ is the strict transform and the order of the indices coincides with the order of the chain, then we have explicit formulas for the multiplicities and log discrepancies.
Let us denote $\X_{m,i}^\infty \coloneqq \X_{m,E_i}^\infty$.

\begin{proposition} \label{prop:contact-m-valuations}
    The contact $m$-valuations of a semihomogeneous singularity of degree $d$ in $\CC^n$ have the following description in terms of the $m$-separating resolution $\mu: X \to \CC^n$:
    \begin{enumerate}[label=(\roman*)]
        \item If $d < n$, then the only contact $m$-valuation is the one given by the divisor $E_{-1}$ (note that this means there are no $m$-valuations if $m < d$).
        \item If $d \geq n$, then every $m$-divisor on $\mu: X \to \mathbb{C}^n$ gives a contact $m$-valuation.
    \end{enumerate}
\end{proposition}

\begin{proof}
    Recall from the proof of Proposition \ref{prop:filtrations-equiv} that $\X_{m,i}^\infty = \{ \gamma \in \X_m^\infty \mid \mathrm{ord}_t \gamma = -i \}.$ By semicontinuity of the order, we conclude that $\overline{\X_{m,i}^\infty} \cap \X_{m,j}^\infty = \varnothing$ if $i < j$.  
    Using \eqref{eq:N-nu-formulas} and \eqref{eq:kappa-r-formula}, the codimension formula of Proposition \ref{prop:codimensions} becomes
    \[
        \mathrm{codim}(\X_{m,i}^\infty) \coloneqq \mathrm{codim}(\X_{m,i}^\infty \hookrightarrow \mathcal{L}_\infty(\mathbb{C}^n)) = m \dfrac{\nu_i}{N_i} = m \dfrac{(m + id) - in}{(m + id) - id} = m + i(d-n).
    \]
    If $d \geq n$, we see that $\mathrm{codim}(\X_{m,i}^\infty) \leq \mathrm{codim}(\X_{m,j}^\infty)$ for every $i < j$, which shows that $\X_{m,i}^\infty \not\subset \overline{\X_{m,j}^\infty}$ for $i < j$. This proves (ii).

    On the other hand, note that $\X_m^\infty$ is defined by $m-d+n$ equations inside $\mathcal{L}_{\infty}(\mathbb{C}^n)$: $n$ equations to impose that $\gamma(0) = 0$ (one for each of the the constant terms of components of the arc $\gamma$) and $m-d$ equations to impose that $\mathrm{ord}_t(f(\gamma(t))) = m$ (one for each of the coefficients of $t^d, \ldots, t^{m-1}$ of $f(\gamma(t))$). 
    Thus any irreducible component of $\X_m^\infty$ has codimension at most $m-d+n$.
    If $d < n$, then for any $j < -1$ we have
    \[
        \mathrm{codim}(\X_{m,j}^\infty) > \mathrm{codim}(\X_{m,-1}^\infty) = m - (d - n)
    \]
    and hence the only irreducible component of $\X_m^\infty$ is $\overline{\X_{m,-1}^\infty}$. This proves (i).
\end{proof}

\begin{remark} 
    The previous proof shows that if $d \geq n$, then $\overline{\X_{m,i}^\infty} \cap \X_{m,j}^\infty = \varnothing$ and $\X_{m,i}^\infty \not\subset \overline{\X_{m,j}^\infty}$ for $i < j$, which implied that each $\overline{\X_{m,i}^\infty}$ is an irreducible component of $\X_m^\infty$.
    Nevertheless, in contrast with what happens for plane curves \cite[Thm. 3.5]{arc-floer-curves}, the irreducible components of $\X_m^\infty$ are \emph{not} disjoint.
    In other words, it may happen that $\X_{m,i}^\infty \cap \overline{\X_{m,j}^\infty} \neq \varnothing$ for $i < j$. This means that the decomposition \eqref{eq:XmE-decomposition} is in general not a stratification.
\end{remark}

\begin{example} \label{ex:counterexample}
    \begin{enumerate}[label=(\roman*), wide]
        \item Consider the polynomial $f = x^4 + y^4 + z^4$ (i.e. $n=3$, $d=4$) and the family of arcs $\gamma_\lambda(t) \coloneqq (t^2, \lambda t, \omega \lambda t)$, where $\omega^4 = -1$.
        We have $\gamma_\lambda \in \X_{8,-1}^\infty$ for all $\lambda \neq 0$, showing that $\gamma_0 \in \X_{8,-2}^\infty \cap \overline{\X_{8,-1}^\infty}$. 
        In particular, $\X_8^\infty(x^4 + y^4 + z^4)$ is connected but it is not irreducible by Proposition~\ref{prop:contact-m-valuations}.

        \item Consider the polynomial $f = x^2 + y^2 + z^2$ (i.e. $n=3$, $d=2$). 
        By Proposition \ref{prop:contact-m-valuations}, there is only one contact $4$-valuation, meaning that $\X_4^\infty$ is irreducible.
        However, given an arc $\gamma(t) = (\sum a_i t^i, \sum b_i t^i, \sum c_i t^i) \in \X_4^\infty$ it is a straightforward computation to show that the Jacobian matrix of the defining equations has maximal rank at $\gamma$ if and only if $(a_1, b_1, c_1) \neq (0,0,0)$, i.e. if an only if $\gamma \notin F_{-2} \X_4^\infty$.
        In particular $\X_4^\infty(x^2 + y^2 + z^2)$ is irreducible but singular.
    \end{enumerate}
\end{example}

Now we move on to studying the dlt $m$-valuations.
We start with a few lemmas that will be useful to compute intersection numbers in the resolution $\mu: X \to \CC^n$.

\begin{lemma} \label{lem:moving-curves}
    Let $W$ be an algebraic variety, $\pi: \mathbb{P}(\mathcal{F}) \to W$ the $\mathbb{P}^1$-bundle associated to a rank 2 vector bundle $\mathcal{F}$, and $\sigma: W \to \mathbb{P}(\mathcal{F})$ a section. Then, for every curve $C \subset \mathbb{P}(\mathcal{F})$, we have an equality of rational cycles
    \[
        [C] = \mathrm{deg}(\pi|_C) \cdot [\sigma(\pi(C))] + \pi^* \alpha_0,
    \]
    for some rational $0$-cycle $\alpha_0 \in A_0(W)$.
\end{lemma}

\begin{proof}
    By \cite[Thm. 3.3]{fulton-intersection}, any $1$-cycle $\beta \in A_1(\mathbb{P}(\mathcal{F}))$ can be written as
    \[
        \beta = (c_1(\mathscr{O}_{\mathbb{P}(\mathcal{F})}(1)) \frown \pi^* \alpha_1) + \pi^* \alpha_0,
    \]
    for unique $\alpha_1 \in A_1(W)$ and $\alpha_0 \in A_0(W)$. 
    Applying $\pi_*$ to both sides, we find $\pi_* \beta = \alpha_1$.
    If we take $\beta = [C] - \mathrm{deg}(\pi|_C) \cdot [\sigma(\pi(C))]$, we have $\pi_* \beta = 0$ and by uniqueness we get the desired result. 
\end{proof}

\begin{lemma} \label{lem:normal-bundle-strict-transform}
    Let $W$ be a smooth variety, $D$ a smooth divisor on $W$, $Z$ a smooth subvariety contained in the support of $D$, and $\beta: \Bl_Z W \to W$ the blow-up of $W$ at $Z$.
    Let $E$ be the exceptional divisor of $\beta$ and let $\widetilde{D}$ be the strict transform of $D$ under $\beta$. Then,
    \[
        \mathcal{N}_{\widetilde{D}/\Bl_Z W} \cong \beta^*\mathcal{N}_{D/W} \otimes \mathscr{O}_{\widetilde{D}}(-E \cap \widetilde{D}).
    \]
\end{lemma}

\begin{proof}
    Note that
    \[
        \beta^* \mathscr{O}_{W}(D) \cong \mathscr{O}_{\Bl_Z W}(\widetilde{D} + E) \cong \mathscr{O}_{\Bl_Z W}(\widetilde{D}) \otimes \mathscr{O}_{\Bl_Z W}(E), 
    \]
    so tensoring by $\mathscr{O}_{\Bl_Z W}(-E)$ and restricting to $\widetilde{D}$ we obtain the desired result, since $\mathcal{N}_{\widetilde{D}/\Bl_Z W} \cong \mathscr{O}_{\Bl_Z W}(\widetilde{D})|_{\widetilde{D}}$, see \cite[21.2.J]{vakil}.
\end{proof}

\begin{lemma} \label{lem:normal-bundle-computation}
    Let $\mu: X \to \mathbb{C}^n$ be the embedded resolution introduced in \S\ref{subsec:m-sep-resolution}, and let $E$ be the divisor associated to the coprime pair $(\kappa, r) \neq (0, 1)$. Let $(\kappa', r')$ and $(\kappa'', r'')$ be the coprime pairs associated to the two divisors whose intersection was blown up to obtain $E$. Assume that $E_{(\kappa', r')}$ is the one closer to $E_{(0,1)}$ so that $E \cong \mathbb{P}(\mathcal{F})$ with $\mathcal{F} = \mathscr{O}_S(\kappa' + r'd) \oplus \mathscr{O}_S(-\kappa'' - r''d)$, see \eqref{eq:normal-bundle-computation}. Let $(\kappa^*, r^*)$ and $(\kappa^{**}, r^{**})$ be the coprime pairs associated to the two divisors adjacent to $E$ in the chain of Figure~\ref{fig:resolution}, with $E_{(\kappa^*, r^*)}$ being the one closer to $E_{(0,1)}$. Let
    \[
        n' \coloneqq \frac{\kappa^* - \kappa'}{\kappa} = \frac{r^* - r'}{r}, \quad
        n'' \coloneqq \frac{\kappa^{**} - \kappa''}{\kappa} = \frac{r^{**} - r''}{r}.
    \]
    Then, the normal bundle of $E$ in $X$ is
    \[
        \mathcal{N}_{E/X} \cong \mathscr{O}_E(-1) \otimes \Big(\mathscr{O}_E(1) \otimes \pi^* \mathscr{O}_S(-\kappa'' - r''d)\Big)^{\otimes -n'} \otimes \Big(\mathscr{O}_E(1) \otimes \pi^* \mathscr{O}_S(\kappa' + r'd)\Big)^{\otimes -n''}.
    \]
\end{lemma}

\begin{proof}
    Our strategy is to study how the normal bundle of $E$ changes after every blow-up in the sequence \eqref{eq:sequence-blow-ups}.
    Recall that we denoted by $E^{(i)} \subset X^{(i)}$ the image of $E$ under the composition $\mu^{(i+1)} \circ \cdots \circ \mu^{(\ell)}$.
    There is a smallest $i_0 \in \{ 1, \ldots, \ell \}$ such that $E^{(i_0)}$ is a divisor on $X^{(i_0)}$.
    Then, by construction, $\mu^{(i_0)}$ is the blow-up of $X^{(i_0 - 1)}$ at the intersection $E_{(\kappa', r')}^{(i_0 - 1)} \cap E_{(\kappa'', r'')}^{(i_0 - 1)}$ and $E^{(i_0)}$ is the exceptional divisor of $\mu^{(i_0)}$.
    Thus we know that $\mathcal{N}_{E^{(i_0)}/X^{(i_0)}} = \mathscr{O}_{E^{(i_0)}}(-1)$ by \cite[22.3.D.(b)]{vakil}.

    The center of every further blow-up $\mu^{(i+1)},\,i \geq i_0$ is either disjoint from $E^{(i)}$, in which case $\mathcal{N}_{E^{(i+1)}/X^{(i+1)}} = (\mu^{(i+1)})^* \mathcal{N}_{E^{(i)}/X^{(i)}}$, or it is the intersection $F \cap E^{(i)}$ of $E^{(i)}$ with one of the two divisors adjacent to it, in which case we can apply Lemma \ref{lem:normal-bundle-strict-transform}.
    All that remains is to compute the invertible sheaf $\mathscr{O}_{E^{(i)}}(-F \cap E^{(i)})$ appearing in said Lemma.

    Let $\mathcal{L}' = \mathscr{O}_S(\kappa' + r'd)$ and $\mathcal{L}'' = \mathscr{O}_S(-\kappa'' - r'' d)$, so $\mathcal{F} = \mathcal{L}' \oplus \mathcal{L}''$.
    Under the isomorphism $E^{(i)} \cong \mathbb{P}(\mathcal{F})$, the intersection of $E^{(i)}$ with the divisor to its left (in the sense of Figure \ref{fig:resolution}) is sent to to the image of the section $\sigma': S \to \mathbb{P}(\mathcal{F})$ defined by the line subbundle $\mathcal{L}' \oplus 0 \subset \mathcal{F}$. 
    Note that the composition
    \[
        \mathscr{O}_{\mathbb{P}(\mathcal{F})}(-1) \to \pi^* \mathcal{F} \to \pi^*(\mathcal{F}/(\mathcal{L}' \oplus 0)) \cong \pi^*(\mathcal{L}'')
    \]
    vanishes precisely at the points of $\mathbb{P}(\mathcal{F})$ given by the one-dimensional subspaces defined by $\mathcal{L}' \oplus 0$, i.e. at the image of the section $\sigma': S \to \mathbb{P}(\mathcal{F})$.
    Hence, seeing $\sigma(S)$ as a divisor on $\mathbb{P}(\mathcal{F})$ we have an isomorphism
    \begin{equation} \label{eq:normal-bundle-adjacent-divisor}
        \mathscr{O}_{\mathbb{P}(\mathcal{F})}(\sigma'(S)) 
        \cong \shom(\mathscr{O}_{\mathbb{P}(\mathcal{F})}(-1), \pi^*(\mathcal{L}'')) 
        \cong \mathscr{O}_{\mathbb{P}(\mathcal{F})}(1) \otimes \pi^*(\mathcal{L}'').
    \end{equation}
    The isomorphism $E^{(i)} \cong \mathbb{P}(\mathcal{F})$ sends $\mathscr{O}_{\mathbb{P}(\mathcal{F})}(\sigma'(S)) $ to $\mathscr{O}_{E^{(i)}}(F \cap E^{(i)})$, where $F$ is the divisor to the left of $E$. A similar argument works for the divisor to the right of $E$.

    All that is left to do is to count how many times during the sequence of blow-ups we have blown up the intersection with divisor on the left of $E$ and how many times the intersection with the divisor on the right.
    Since the divisor to the left of $E$ in $X^{(i_0)}$ is $E_{(\kappa', r')}$, after blowing up the intersection of $E$ with the divisor to the left $n'$ times, the resulting divisor on the left of $E$ has associated coprime pair $(\kappa' + n' \kappa, r' + n' r)$ (indeed, at the beginning the pair is $(\kappa', r')$ and each blow-up adds $(\kappa, r)$).
    Since we are denoting by $(\kappa^*, r^*)$ the pair associated to the divisor to the left of $E$ in $X$, we obtain the desired formula for $n'$.
    The analogous argument gives the formula for $n''$.
\end{proof}

\begin{remark}
    The coprime pairs $(\kappa', r')$ and $(\kappa'', r'')$ are not extra data that one needs to remember from the resolution process: they may be easily obtained from the continued fraction expansion of $\kappa/r$. Indeed, if $\kappa/r = [q_0; q_1, \ldots, q_k]$, then
    \[
        \kappa'/r' = [q_0; q_1, \ldots, q_{k-1}] \quad \text{ and } \quad
        \kappa''/r'' = [q_0; q_1, \ldots, q_{k-1}, q_k - 1].
    \]
\end{remark}

\begin{proposition} \label{prop:dlt-m-valuations}
    The dlt $m$-valuations of a semihomogeneous singularity of degree $d$ in $\CC^n$ have the following description in terms of the $m$-separating resolution $\mu: X \to \CC^n$:
    \begin{enumerate}[label=(\roman*)]
        \item If $d < n$, there are no dlt $m$-valuations.
        \item If $d \geq n$, then every $m$-divisor on $\mu: X \to \mathbb{C}^n$ gives a dlt $m$-valuation.
    \end{enumerate}
\end{proposition}

\begin{proof}
    Recall that $D=\mathrm{div}(f)$. 
    Let $\Delta \coloneqq \mu^*(D)_{\mathrm{red}}$.
    By \cite[Lemma 3.4, Lemma 3.5]{irred-curve}, to find all the dlt $m$-valuations we just need to study which $m$-divisors appear on the dlt modifications of $(\mathbb{C}^n,D)$ that are minimal models over $\mathbb{C}^n$ of a particular $m$-separating log resolution; in our case we choose $(X,\Delta)$, see \cite[Def. 2.4.(iii)]{irred-curve}. 
    In particular, (i) is equivalent to the statement that $(\mathbb{C}^n, D)$ is a minimal model of $(X, \Delta)$ over $(\mathbb{C}^n, D)$ for $d < n$, while (ii) is equivalent to the statement that $(X, \Delta)$ is its own minimal model over $(\mathbb{C}^n, D)$ whenever $d \geq n$.
    
    Recall that for any exceptional $E \in \mathcal{E}$, its \emph{log discrepancy with respect to $(\mathbb{C}^n, D)$}
    is $\nu_E - N_E$, see \cite[Def. 2.1]{irred-curve}.
    Therefore, $(\mathbb{C}^n, D)$ is a minimal model of $(X, \Delta)$ over $(\mathbb{C}^n, D)$ if and only if we have the inequality $N_E < \nu_E$ for every exceptional $E \in \mathcal{E}$. By \eqref{eq:N-nu-formulas}, this happens if and only if $d < n$. This proves (i).

    Now suppose $d \geq n$. 
    To show that $(X, \Delta)$ is its own minimal model over $(\mathbb{C}^n, D)$, we need to see that $K_X + \Delta = \sum_{E \in \mathcal{E}} \nu_E E$ is $\mu$-nef, i.e. that $(K_X + \Delta) \cdot C \geq 0$ for every curve $C \subset X$ contracted by $\mu$.

    Suppose first that $C$ is an irreducible curve contained in $E_{(0,1)}$. Note that the unique divisor $E_{(\kappa, r)}$ such that $E_{(0,1)} \cap E_{(\kappa, r)} \neq \varnothing$ necessarily has $\kappa = 1$ in its coprime pair. Then,
    \begin{align*}
        (K_X + \Delta) \cdot C &= \mathrm{deg}_C \mathscr{O}_X(\nu_{(0,1)} E_{(0,1)} + \nu_{(1,r)} E_{(1,r)}) \\
        &\stackrel{\text{\eqref{eq:N-nu-formulas}}}{=} n \cdot \mathrm{deg}_C \mathscr{O}_X(E_{(0,1)}) + (1+rn) \cdot \mathrm{deg}_C \mathscr{O}_X(E_{(1,r)}) \\
        &\stackrel{\phantom{a}}{=} n \cdot \mathrm{deg}_C \mathscr{O}_{\mathbb{P}^{n-1}}(-1) + (1+rn) \cdot \mathrm{deg}_C \mathscr{O}_{\mathbb{P}^{n-1}}(d) \\
        &\stackrel{\phantom{a}}{=} (-n + d + rnd) \cdot \mathrm{deg}_C \mathscr{O}_{\mathbb{P}^{n-1}}(1) \geq 0.
    \end{align*}
    In the first equality we used that $C$ does not intersect any divisors in $\mathcal{E}$ other than $E_{(0,1)}$ and $E_{(1,r)}$; in the second to last equality we used the isomorphism $E_{(0,1)} \cong \mathbb{P}^{n-1}$ and the induced isomorphisms
    \begin{align*}
        \mathscr{O}_X(E_{(0,1)})|_{E_{(0,1)}} &\cong \mathcal{N}_{E_{(0,1)}/X} \cong \mathscr{O}_{\mathbb{P}^{n-1}}(-1), \\
        \mathscr{O}_X(E_{(1,r)})|_{E_{(0,1)}} &\cong \mathscr{O}_{E_{(0,1)}}(E_{(0,1)} \cap E_{(1,r)}) \cong \mathscr{O}_{\mathbb{P}^{n-1}}(d).
    \end{align*}
    We know that the last line is non-negative because $d \geq n > 0$, $r \geq 0$ and the degree of a curve is non-negative.

    Now suppose $C$ is an irreducible curve contained in $E_{(\kappa,r)}$ for $(\kappa,r) \neq (0,1), (1,0)$. 
    Recall from the proof of Proposition \ref{prop:exceptional-divisor-topology} that $E_{(\kappa,r)}$ is the projectivization $\mathbb{P}(\mathcal{F})$ of the rank 2 vector bundle $\mathcal{F} = \mathcal{L}' \oplus \mathcal{L}''$ on $S$, where $\mathcal{L}' \cong O_S(-\kappa' - r'd)$ and $\mathcal{L}'' \cong O_S(\kappa'' + r'' d)$.
    These line subbundles of $\mathcal{F}$ define two sections $\sigma', \sigma'': S \to \mathbb{P}(\mathcal{F})$, whose images correspond to the intersection of $E_{(\kappa, r)}$ with the adjacent divisors in the chain of Figure~\ref{fig:resolution}.

    By Lemma \ref{lem:moving-curves}, we may write the rational equivalence class of $C$ as $[C] = \mathrm{deg}(\pi|_C) \cdot [\sigma'(\pi(C))] + \pi^* \alpha_0$ where $\alpha_0 \in A_0(S)$ and $\mathrm{deg}(\pi|_C) \geq 0$. 
    Since $\sigma'(\pi(C))$ is contained in a divisor adjacent to $E_{(\kappa,r)}$, we may apply this argument repeatedly to write $[C]$ as a the sum of a positive multiple of the class of a curve $C'$ contained in $E_{(0,1)}$ and an integer combination of fibers of the different $E_{(\kappa,r)}$ with $(\kappa, r) \neq (0,1)$.
    Since we already know that $(K_X + \Delta) \cdot C' \geq 0$ by the previous paragraph, it suffices to show that $(K_X + \Delta) \cdot F = 0$ for every fiber $F$ of the projection $\pi: E_{(\kappa,r)} \to S$ (note that it is \emph{not} enough to show $(K_X + \Delta) \cdot F \geq 0$ because its coefficient in the expression of $[C]$ might be negative).

    Hence, let $F$ be a fiber of $\pi: E_{(\kappa,r)} \to S$ and let $(\kappa^*, r^*)$ and $(\kappa^{**}, r^{**})$ be the coprime pairs associated to the two divisors adjacent to $E \coloneqq E_{(\kappa, r)}$ in the chain of Figure~\ref{fig:resolution}, with $E_{(\kappa^*, r^*)}$ being the one closer to $E_{(0,1)}$. We have
    \begin{align*}
        \mathscr{O}_X(E)|_E \cong \mathcal{N}_{E/X} &\cong \mathscr{O}_E(-1-n'-n'') \otimes \pi^* \mathscr{O}_S(a), \quad \text{for some $a \in \mathbb{Z}$} \\
        \mathscr{O}_X(E_{(\kappa^*, r^*)})|_E &\cong \mathscr{O}_E(1) \otimes \pi^* \mathscr{O}_S(\kappa^* + r^* d), \\
        \mathscr{O}_X(E_{(\kappa'', r'')})|_E &\cong \mathscr{O}_E(1) \otimes \pi^* \mathscr{O}_S(-\kappa'' - r'' d),
    \end{align*}
    where the first line comes from Lemma~\ref{lem:normal-bundle-computation} and the second and third lines come from \eqref{eq:normal-bundle-adjacent-divisor}. We have $\mathrm{deg}_F \mathscr{O}_E(1) = 1$ and $\mathrm{deg}_F \pi^* \mathscr{O}_S(k) = 0$ for every $k \in \mathbb{Z}$. Then,
    \begin{align*}
        (K_X + \Delta) \cdot F &= (\nu_{(\kappa^*, r^*)} E_{(\kappa^*, r^*)} + \nu_{(\kappa, r)} E_{(\kappa, r)} + \nu_{(\kappa'', r'')} E_{(\kappa'', r'')}) \cdot F \\
        &= \nu_{(\kappa^*, r^*)} + \nu_{(\kappa, r)} (-1 - n' - n'') + \nu_{(\kappa'', r'')} \\
        &= (\kappa^* + r^* n) - (\kappa' + r' n)(1 + n' + n'') - (\kappa'' + r'' n) = 0.
    \end{align*}
\end{proof}

Finally, we move on to the essential $m$-valuations.

\begin{proposition} \label{prop:minimal-m-separating-log-res}
    The map $\mu: X \to \mathbb{C}^n$ is the minimal $m$-separating log resolution of the triple $(\mathbb{C}^n, f, 0)$, in the sense that any other $m$-separating log resolution factors through $\mu$.
\end{proposition}

\begin{proof}
    Let $\nu: Y \to \mathbb{C}^n$ be an arbitrary $m$-separating log resolution of the triple $(\mathbb{C}^n, f, 0)$.
    We are going to see that $\nu$ factors through each of the blow-ups in the sequence \eqref{eq:sequence-blow-ups} that defines $\mu$.

    By definition of log resolution of a triple, we know that $\nu^{-1}(0)$ is a normal crossing divisor. 
    Hence, by the universal property of the blow-up, $\nu$ factors through the blow-up of $\mathbb{C}^n$ at the origin, i.e. through $\mu^{(1)}$.
    
    Now suppose that we have already shown that $\nu$ factors through $\mu^{(i)}$ for some $i < \ell$.
    We want to apply once again the universal property of the blow-up to conclude that $\nu$ factors through $\mu^{(i+1)}$, which is the blow-up of $X^{(i)}$ at the intersection of two divisors $E$ and $F$. 
    To be able to do that, we just need to show that $\nu^{-1}(E \cap F)$ is a divisor on $Y$.

    Suppose that $\nu^{-1}(E \cap F)$ is not a divisor on $Y$, so there is some irreducible component $C$ of $\nu^{-1}(E \cap F)$ with $\mathrm{codim}_Y C \geq 2$.
    Then there is some point $y \in C$ such that the restriction of $\nu$ to a neighborhood of $y$ is an isomorphism (this is Zariski's Main Theorem, see \cite[Prop. III.9.1]{mumford1999}).
    In particular, that the strict transforms of $E$ and $F$ on $Y$ intersect at $\nu(y)$.
    
    On the other hand, we know that $N_E + N_F \leq m$, since otherwise we would not have blown up the intersection of $E$ and $F$ in the construction of $\mu$.
    Hence, since $\nu$ is $m$-separating, the strict transforms of $E$ and $F$ on $Y$ (which exist because we are assuming that $\nu$ factors through $\mu^{(i)}$) do not intersect. This is a contradiction, so $\nu^{-1}(E \cap F)$ must be a divisor on $Y$ and we are done.
\end{proof}

\begin{corollary} \label{cor:essential-m-valuations}
    The essential $m$-valuations of a semihomogeneous singularity of degree $d$ in $\CC^n$ are precisely the ones given by $m$-divisors on the $m$-separating resolution $\mu: X \to \CC^n$.    
\end{corollary}






\printbibliography

@article{mclean,
  author    = {M. McLean},
  title     = {{Floer cohomology, multiplicity and the log canonical threshold}},
  volume    = {23},
  journal   = {Geometry \& Topology},
  number    = {2},
  publisher = {MSP},
  pages     = {957-1056},
  year      = {2016},
  doi       = {10.2140/gt.2019.23.957}
}

@article{arc-floer-curves,
  title         = {The Arc-Floer conjecture for plane curves},
  author        = {Javier de la Bodega and Eduardo de Lorenzo Poza},
  year          = {2023},
  eprint        = {2308.00051},
  archiveprefix = {arXiv},
  primaryclass  = {math.AG},
  url           = {https://arxiv.org/abs/2308.00051}
}

@article{coho-contact,
  author   = { Budur, N. and Fern{\'a}ndez de Bobadilla, J. and L{\^e}, {Q. T}.  and Nguyen, {H. D}.},
  label    = {B+},
  title    = {Cohomology of contact loci},
  fjournal = {Journal of Differential Geometry},
  journal  = {J. Differ. Geom.},
  volume   = {120},
  number   = {3},
  pages    = {389--409},
  year     = {2022},
  doi      = {10.4310/jdg/1649953456},
  zbmath   = {7523602},
  zbl      = {1495.14026}
}

@article{irred-curve,
  author   = {Budur, Nero and de la Bodega, Javier and de Lorenzo Poza, Eduardo and Fern{\'a}ndez de Bobadilla, Javier and Pe{\l}ka, Tomasz},
  title    = {On the embedded {Nash} problem},
  fjournal = {Forum of Mathematics, Pi},
  journal  = {Forum Math. Pi},
  volume   = {12},
  pages    = {28},
  note     = {Id/No e15},
  year     = {2024},
  doi      = {10.1017/fmp.2024.13},
  zbmath   = {7939752}
}

@article{bobadilla2024,
  author   = {Fern{\'a}ndez de Bobadilla, Javier and Pe{\l}ka, Tomasz},
  title    = {Symplectic monodromy at radius zero and equimultiplicity of {{\(\mu\)}}-constant families},
  fjournal = {Annals of Mathematics. Second Series},
  journal  = {Ann. Math. (2)},
  volume   = {200},
  number   = {1},
  pages    = {153--299},
  year     = {2024},
  doi      = {10.4007/annals.2024.200.1.4},
  zbmath   = {7995671}
}

@article{Uljarevic,
  author  = {Uljarevic, I.},
  year    = {2017},
  pages   = {861-903},
  title   = {{F}loer homology of automorphisms of {L}iouville domains},
  volume  = {15},
  journal = {Journal of Symplectic Geometry},
  doi     = {10.4310/JSG.2017.v15.n3.a9}
}

@article{jet-arc,
  author     = {Ishii, Shihoko},
  title      = {Jet schemes, arc spaces and the {N}ash problem},
  journal    = {C. R. Math. Acad. Sci. Soc. R. Can.},
  fjournal   = {Comptes Rendus Math\'ematiques de l'Acad\'emie des Sciences.
                La Soci\'et\'e{} Royale du Canada. Mathematical Reports of the
                Academy of Science. The Royal Society of Canada},
  volume     = {29},
  year       = {2007},
  number     = {1},
  pages      = {1--21},
  mrclass    = {14D15 (14D20 14J17)},
  mrnumber   = {2354631},
  mrreviewer = {Tommaso\ De Fernex}
}

@article{DL-motivic,
  author   = {Denef, Jan and Loeser, Fran{\c{c}}ois},
  title    = {Motivic {Igusa} zeta functions},
  fjournal = {Journal of Algebraic Geometry},
  journal  = {J. Algebr. Geom.},
  volume   = {7},
  number   = {3},
  pages    = {505--537},
  year     = {1998},
  zbmath   = {1353475},
  zbl      = {0943.14010}
}

@article{DL-lefnum,
  author   = {Denef, Jan and Loeser, Fran{\c{c}}ois},
  title    = {Lefschetz numbers of iterates of the monodromy and truncated arcs},
  fjournal = {Topology},
  journal  = {Topology},
  volume   = {41},
  number   = {5},
  pages    = {1031--1040},
  year     = {2002},
  doi      = {10.1016/S0040-9383(01)00016-7},
  zbmath   = {1801890},
  zbl      = {1054.14003}
}

@incollection{seidel,
  author     = {Seidel, Paul},
  title      = {More about vanishing cycles and mutation},
  booktitle  = {Symplectic geometry and mirror symmetry ({S}eoul, 2000)},
  pages      = {429--465},
  publisher  = {World Sci. Publ., River Edge, NJ},
  year       = {2001},
  isbn       = {981-02-4714-1},
  mrclass    = {53D40 (18E30)},
  mrnumber   = {1882336},
  mrreviewer = {Mikhail\ G.\ Khovanov},
  doi        = {10.1142/9789812799821\_0012}
}

@article{arapura2005,
  author   = {Arapura, Donu},
  title    = {The {Leray} spectral sequence is motivic},
  fjournal = {Inventiones Mathematicae},
  journal  = {Invent. Math.},
  volume   = {160},
  number   = {3},
  pages    = {567--589},
  year     = {2005},
  doi      = {10.1007/s00222-004-0416-x},
  zbmath   = {2175702},
  zbl      = {1083.14011}
}

@misc{stacks-project,
  shorthand = {Stacks},
  author    = {{The Stacks Project Authors}},
  title     = {\textit{Stacks Project}},
  url       = {https://stacks.math.columbia.edu},
  year      = {2018}
}

@book{dimca1992,
  author     = {Dimca, Alexandru},
  title      = {Singularities and topology of hypersurfaces},
  series     = {Universitext},
  publisher  = {Springer-Verlag, New York},
  year       = {1992},
  pages      = {xvi+263},
  isbn       = {0-387-97709-0},
  mrclass    = {32Sxx (14J70 32S25 32S50 57M25 57R45 58C27)},
  mrnumber   = {1194180},
  mrreviewer = {Aleksandr\ G.\ Aleksandrov},
  doi        = {10.1007/978-1-4612-4404-2}
}

@article{ein2004,
  title     = {Contact loci in arc spaces},
  volume    = {140},
  doi       = {10.1112/S0010437X04000429},
  number    = {5},
  journal   = {Compositio Mathematica},
  publisher = {London Mathematical Society},
  author    = {Ein, L. and Lazarsfeld, R. and Mustaţa, M.},
  year      = {2004},
  pages     = {1229–1244}
}

@book{milnor1974,
  author     = {Milnor, John W. and Stasheff, James D.},
  title      = {Characteristic classes},
  series     = {Annals of Mathematics Studies},
  volume     = {No. 76},
  publisher  = {Princeton University Press, Princeton, NJ},
  year       = {1974},
  pages      = {vii+331},
  mrclass    = {57-01 (55-02 55F40 57D20)},
  mrnumber   = {440554},
  mrreviewer = {F.\ Hirzebruch}
}

@misc{qiaochuyuan,
  title  = {Hypersurfaces, 4-manifolds, and characteristic classes},
  author = {Yuan, Qiaochu},
  url    = {https://qchu.wordpress.com/2014/06/16/hypersurfaces-4-manifolds-and-characteristic-classes/},
  note   = {(version: 2014-06-16)}
}

@misc{MO-Lefschetz-map,
  title        = {Lefschetz map in the middle cohomology of a smooth projective hypersurface},
  author       = {Will Sawin},
  label        = {Saw25},
  howpublished = {MathOverflow},
  url          = {https://mathoverflow.net/q/486487},
  note         = {(version: 2025-01-23)}
}

@book{hartshorne,
  author     = {Hartshorne, Robin},
  title      = {Algebraic geometry},
  series     = {Graduate Texts in Mathematics},
  volume     = {No. 52},
  publisher  = {Springer-Verlag, New York-Heidelberg},
  year       = {1977},
  pages      = {xvi+496},
  isbn       = {0-387-90244-9},
  mrclass    = {14-01},
  mrnumber   = {463157},
  mrreviewer = {Robert\ Speiser}
}

@book{fulton-intersection,
  author    = {Fulton, William},
  title     = {Intersection theory},
  edition   = {2nd ed.},
  publisher = {Springer-Verlag, Berlin},
  year      = {1998},
  pages     = {xiv+470},
  isbn      = {3-540-62046-X; 0-387-98549-2},
  mrclass   = {14C17 (14-02)},
  mrnumber  = {1644323},
  doi       = {10.1007/978-1-4612-1700-8}
}

@misc{MO-fiberbundlecomposition,
  title        = {Composition of covering map and bundle projection},
  author       = {Torsten Ekedahl},
  label        = {Ek11},
  howpublished = {MathOverflow},
  note         = {(version: 2011-08-25)},
  url          = {https://mathoverflow.net/q/73546}
}

@book{vakil,
  author    = {Vakil, Ravi},
  title     = {The rising sea. {Foundations} of algebraic geometry (to appear)},
  isbn      = {978-0-691-26866-8; 978-0-691-26867-5; 978-0-691-26868-2},
  year      = {2025},
  publisher = {Princeton, NJ: Princeton University Press},
  language  = {English},
  zbmath    = {7961837}
}

@book{milnor1968,
  author     = {Milnor, John},
  title      = {Singular points of complex hypersurfaces},
  series     = {Annals of Mathematics Studies},
  volume     = {No. 61},
  publisher  = {Princeton University Press, Princeton, NJ},
  year       = {1968},
  pages      = {iii+122},
  mrclass    = {57.20 (14.00)},
  mrnumber   = {239612},
  mrreviewer = {J.\ P.\ Levine}
}

@article{dostoglou-salamon94,
  author     = {Dostoglou, Stamatis and Salamon, Dietmar A.},
  title      = {Self-dual instantons and holomorphic curves},
  journal    = {Ann. of Math. (2)},
  fjournal   = {Annals of Mathematics. Second Series},
  volume     = {139},
  year       = {1994},
  number     = {3},
  pages      = {581--640},
  mrclass    = {58E15 (53C07 57R57 58D15 58D27)},
  mrnumber   = {1283871},
  mrreviewer = {S.\ Donaldson},
  doi        = {10.2307/2118573}
}

@article{seade2019,
  author   = {Seade, Jos{\'e}},
  title    = {On {Milnor}'s fibration theorem and its offspring after 50 years},
  fjournal = {Bulletin of the American Mathematical Society. New Series},
  journal  = {Bull. Am. Math. Soc., New Ser.},
  volume   = {56},
  number   = {2},
  pages    = {281--348},
  year     = {2019},
  doi      = {10.1090/bull/1654},
  zbmath   = {7045323},
  zbl      = {1416.32001}
}

@book{chambertloir2018,
  author    = {Chambert-Loir, Antoine and Nicaise, Johannes and Sebag, Julien},
  title     = {Motivic integration},
  fseries   = {Progress in Mathematics},
  series    = {Prog. Math.},
  volume    = {325},
  isbn      = {978-1-4939-7885-4; 978-1-4939-9315-4},
  year      = {2018},
  publisher = {New York, NY: Birkh{\"a}user},
  doi       = {10.1007/978-1-4939-7887-8},
  zbmath    = {6862764},
  zbl       = {1545.14021}
}

@article{contact-hyperplane,
  author   = {Budur, Nero and Tue, Tran Quang},
  title    = {On contact loci of hyperplane arrangements},
  fjournal = {Advances in Applied Mathematics},
  journal  = {Adv. Appl. Math.},
  volume   = {132},
  pages    = {28},
  note     = {Id/No 102271},
  year     = {2022},
  doi      = {10.1016/j.aam.2021.102271},
  zbmath   = {7430808},
  zbl      = {1483.14028}
}

@article{bhatt2016,
  author   = {Bhatt, Bhargav},
  title    = {Algebraization and {Tannaka} duality},
  fjournal = {Cambridge Journal of Mathematics},
  journal  = {Camb. J. Math.},
  volume   = {4},
  number   = {4},
  pages    = {403--461},
  year     = {2016},
  doi      = {10.4310/CJM.2016.v4.n4.a1},
  zbmath   = {6670053},
  zbl      = {1356.14006}
}

@article{ishii2004,
  author   = {Ishii, Shihoko},
  title    = {The arc space of a toric variety},
  fjournal = {Journal of Algebra},
  journal  = {J. Algebra},
  volume   = {278},
  number   = {2},
  pages    = {666--683},
  year     = {2004},
  doi      = {10.1016/j.jalgebra.2003.12.015},
  zbmath   = {2105350},
  zbl      = {1073.14066}
}

@article{delabodega2025,
  title         = {The embedded {N}ash problem in singular spaces: the case of surfaces},
  author        = {Javier de la Bodega},
  year          = {2025},
  eprint        = {2408.01533},
  archiveprefix = {arXiv},
  primaryclass  = {math.AG},
  url           = {https://arxiv.org/abs/2408.01533}
}

@article{defernex2016,
  author   = {de Fernex, Tommaso and Docampo, Roi},
  title    = {Terminal valuations and the {Nash} problem},
  fjournal = {Inventiones Mathematicae},
  journal  = {Invent. Math.},
  volume   = {203},
  number   = {1},
  pages    = {303--331},
  year     = {2016},
  doi      = {10.1007/s00222-015-0597-5},
  zbmath   = {6541953},
  zbl      = {1345.14020}
}

@book{mumford1999,
  author    = {Mumford, David},
  title     = {The red book of varieties and schemes. {Includes} the {Michigan} lectures (1974) on ``{Curves} and their {Jacobians}''.},
  edition   = {2nd, expanded ed. with contributions by {Enrico} {Arbarello}},
  fseries   = {Lecture Notes in Mathematics},
  series    = {Lect. Notes Math.},
  volume    = {1358},
  isbn      = {3-540-63293-X},
  year      = {1999},
  publisher = {Berlin: Springer},
  language  = {English},
  doi       = {10.1007/b62130},
  zbmath    = {1252431},
  zbl       = {0945.14001}
}

@book{spanier1966,
  author     = {Spanier, Edwin H.},
  title      = {Algebraic topology},
  publisher  = {McGraw-Hill Book Co., New York-Toronto-London},
  year       = {1966},
  pages      = {xiv+528},
  mrclass    = {55.00},
  mrnumber   = {210112},
  mrreviewer = {S.-T.\ Hu}
}

@book{arnold1985,
  author    = {Arnol'd, V. I. and Guse\u in-Zade, S. M. and Varchenko,
               A. N.},
  title     = {Singularities of differentiable maps. {V}ol. {I}},
  series    = {Monographs in Mathematics},
  volume    = {82},
  note      = {The classification of critical points, caustics and wave
               fronts,
               Translated from the Russian by Ian Porteous and Mark Reynolds},
  publisher = {Birkh\"auser Boston, Inc., Boston, MA},
  year      = {1985},
  pages     = {xi+382},
  isbn      = {0-8176-3187-9},
  mrclass   = {58C27},
  mrnumber  = {777682},
  doi       = {10.1007/978-1-4612-5154-5},
  url       = {https://doi.org/10.1007/978-1-4612-5154-5}
}

\noindent
\textsc{Eduardo de Lorenzo Poza}. \url{eduardo.delorenzopoza@kuleuven.be} \\
Basque Center for Applied Mathematics. Alameda Mazarredo 14, 48009 Bilbao, Spain. \\
KU Leuven, Department of Mathematics. Celestijnenlaan 200B, 3001 Heverlee, Belgium.

\noindent
\textsc{Jiahui Huang}. 
\url{j346huan@uwaterloo.ca} \\
University of Waterloo.
Department of Pure Mathematics. \\
200 University Ave W.,
Waterloo, ON N2L 3G1,
Canada.

\end{document}